\documentclass[12pt]{article}

\title{Online adaptive basis enrichment for mixed CEM-GMsFEM}
\author{Eric Chung\thanks{Department of Mathematics, The Chinese University of Hong Kong, Shatin, Hong Kong}, ~Sai-Mang Pun\thanks{Department of Mathematics, The Chinese University of Hong Kong, Shatin, Hong Kong}}

\usepackage{amscd,amssymb}
\usepackage[latin1]{inputenc}
\usepackage{tikz}
\usetikzlibrary{shapes.geometric,arrows}

\usepackage{pifont}
\usepackage{amsmath}
\usepackage{graphicx}
\usepackage{subfigure,fancybox}
\usepackage{epic,eepic}
\usepackage{amstext}
\usepackage{pstricks}
\usepackage{pst-node,pst-plot,pst-coil}
\usepackage{amsthm}
\usepackage[english]{babel}
\usepackage[square,numbers]{natbib}
\usepackage{booktabs}
\usepackage{hyperref}
\newcommand\diff{\,\mathrm{d}}
\newcommand\norm[1]{\lVert#1\rVert}
\newcommand\abs[1]{\lvert#1\rvert}

\newcommand\Norm[1]{\bigg\Vert#1\bigg\Vert}
\theoremstyle{definition}
\newtheorem{theorem}{Theorem}[section]

\newtheorem{lemma}[theorem]{Lemma}
\newtheorem{example}[theorem]{Example}

\theoremstyle{remark}
\newtheorem*{remark}{Remark}
\newcommand{\revi}[1]{\textcolor{black}{#1}}
\newcommand{\revii}[1]{\textcolor{black}{#1}}
\newcommand{\reviii}[1]{\textcolor{black}{#1}}
\newcommand\divg[1]{\nabla \cdot #1}

\begin{document}

\maketitle
\begin{abstract}
In this research, an online basis \reviii{enrichment strategy} for the constraint energy minimizing generalized multiscale finite element method in mixed formulation is proposed. The online approach is based on the \reviii{technique} of oversampling. One makes use of the information of residual and the \reviii{data} in the partial differential equation such as the source function. The analysis presented shows that the proposed online enrichment leads to a fast convergence from multiscale approximation to the fine-scale solution. The error reduction can be made sufficiently large by suitably selecting oversampling regions and the number of oversampling layers. \reviii{Also, the convergence rate of the enrichment can be tuned by a user-defined parameter.} Numerical results are provided to illustrate the efficiency of the proposed method.
\end{abstract}

\section{Introduction} \label{sec:intro}

Many problems arising from physics and engineering involve multiple scales and high contrast, such as \reviii{Darcy's flow model} in heterogeneous porous media. \revi{These problems are} prohibitively costly to solve when traditional fine-scale solvers are directly applied and some type of reduced-order models are considered to avoid the high computational cost. Many model reduction techniques have been well developed in the existing literature. For example, in upscaling methods \cite{chen2003coupled, durlofsky1991numerical, wu2002analysis} which are commonly used, one typically derives another {\it upscaling} media and solves the upscaled problem globally on a coarse grid. This can be done by solving local problems in each coarse element, namely, computing the effective permeability field. 


\reviii{Besides the upscaling approaches mentioned above, multiscale methods  \cite{chung2016adaptive,efendiev2009multiscale,efendiev2000convergence,hou1997multiscale,hughes1995multiscale} have been widely used to approximate the solution of the multiscale problem.} In multiscale methods, the solution of the problem is approximated by local basis functions, which are solutions to a class of local problems on coarse grid. Moreover, because of the necessity of the mass conservation for velocity fields, \revi{many approaches have been proposed} to guarantee this property, such as multiscale finite volume methods  \cite{cortinovis2014iterative,hajibeygi2011hierarchical,jenny2003multi,lie2017feature,lunati2004multi}, mixed multiscale finite element methods \cite{aarnes2004use,aarnes2008mixed,chan2016adaptive,chen2016least,chen2003mixed,chung2015mixed,chung2016mixed}, mortar multiscale methods \cite{arbogast2007multiscale,chung2016enriched,peszynska2005mortar,peszynska2002mortar} and various postprocessing methods \cite{bush2013application,odsaeter2017postprocessing}. 


Preserving the property of mass conservation in multiscale simulations requires special formulation. In a mixed finite element formulation, one may consider a first-order system for pressure and velocity. \reviii{In order to approximate the velocity field, multiscale basis functions are constructed by solving a class of local cell problems with Neumann boundary conditions. One may use a piecewise constant to approximate the pressure field.} 
In this case, the support of pressure basis function consists of a single coarse block, while the support of velocity basis function consists of more than one coarse block sharing a common interface. These multiscale approaches have been well developed in \cite{aarnes2004use,aarnes2008mixed,chen2003mixed} and applied to various situations. 

Recently, a generalized multiscale finite element method (GMsFEM) in mixed formulation has been proposed in \cite{chung2015mixed,efendiev2013mini}. \reviii{The GMsFEM provides} a systematic procedure to construct multiple basis functions for either velocity or pressure in each local patch, which makes these methods different from previous methodology in applications. The computation of velocity basis functions involves a construction of snapshot space and a model reduction via local spectral decomposition to identify appropriate modes to form the multiscale space. The convergence analysis in \cite{chung2015mixed} addresses a spectral convergence $1/\Lambda$, where $\Lambda$ is the smallest eigenvalue whose modes are excluded in the multiscale space. 
\reviii{In \cite{chung2018constraintmixed}}, a variation of GMsFEM based on a constraint energy minimization (CEM) strategy \cite{chung2018constraint} for mixed formulation has been \reviii{developed, providing} a better convergence rate proportional to $H/\Lambda$ with $H$ the size of coarse mesh. This approach makes use of the ideas of oversampling and localization \cite{maalqvist2014localization,owhadi2017multigrid,owhadi2014polyharmonic} to compute multiscale basis functions in oversampled subregions with the satisfaction of an appropriate orthogonality condition. The proposed method provides a mass conservative velocity field and allows one to identify some nonlocal information \reviii{depending on the contrast of the permeability field. This is done by solving a class of well-designed local spectral problems to construct the multiscale space.}


\reviii{The construction of multiscale space can be regarded as {\it offline} computation since it does not take into account the source term and this procedure can be done before solving the actual multiscale approximation. 
In the offline stage, the multiscale solver} can be tuned in various ways to achieve smaller errors; however, the error decay slows down when a certain number of degrees of freedom is reached. This is due to some slow decay after certain eigenvalues. \reviii{Efendiev and co-authors in \cite{chung2015residual,chung2017online} proposed an {\it online} construction within the framework of GMsFEM in order to enhance the accuracy of multiscale approximation. Based on the information of residual related to a computed coarse-grid approximation, one may construct new basis functions with local support in the online stage of simulation.}
\reviii{For the diffusion problem,} the analysis in \cite{chung2015residual,chung2017online} shows that the error decay is proportional to $1-C\Lambda$, where $C$ is a constant independent of scales and contrast and guarantees the positivity of the convergence rate.

In this research, we \reviii{propose an online enrichment strategy} for CEM-GMsFEM in mixed formulation. The strategy is based on the residual information and the technique of oversampling, adopting the ideas in \cite{chung2018fast}, and \reviii{the corresponding online basis functions are supported} in an oversampled region. \reviii{This construction differs from the previous online approach in \cite{chan2016adaptive} since CEM-GMsFEM makes use of the technique of oversampling. In particular, the online basis functions are formulated in the oversampled regions. We show that the proposed online adaptive method provides a better analytical convergence rate compared to $1- \Lambda$ in online mixed GMsFEM \cite{chan2016adaptive}, which additionally requires that the online basis for velocity is divergence free. One can obtain very accurate approximation in one online iteration by choosing an appropriate number of (oversampling) layers. 
In particular, one may achieve a fast convergency to the fine-scale solution with the number of oversampling layers large enough if sufficiently many offline basis functions are included in the construction of initial multiscale space}.

The paper is organized as follows. In Section \ref{sec:prelim}, we present the preliminaries of the model problem. Next, we describe the framework of CEM-GMsFEM in Section \ref{sec:multiscale}. The online adaptive algorithm and the analysis are presented in Section \ref{sec:online}. In Section \ref{sec:numerics}, the numerical results are provided to illustrate the efficiency of the proposed method. Concluding remarks will be drawn in Section \ref{sec:conclusion}. 

\section{Preliminaries}\label{sec:prelim}
\reviii{Consider} a class of high contrast flow problems in the following mixed formulation over the computational domain $\Omega \subset \mathbb{R}^d$ ($d= 2,3$):
\begin{align}
     \kappa^{-1} v + \nabla p & = 0 \quad \text{in } \Omega,\label{eqn:model_1}\\
     \divg{v}  & = f \quad \text{in } \Omega, \label{eqn:model_2} \\
     v \cdot \mathbf{n} & = g \quad \text{on } \partial \Omega, \label{eqn:model_bc}  \\
     \int_\Omega p \diff{x} & = 0,
\end{align}
where $\mathbf{n}$ is the outward unit normal vector field on the boundary $\partial \Omega$. Note that the source function $f\in L^2(\Omega)$ and the function $g \in L^2(\partial \Omega)$ satisfy the following compatibility condition: 
$$ \int_\Omega f \diff{x} = \int_{\partial \Omega} g \diff{S(x)}.$$ 
Assume that the function $\kappa: \Omega \to \mathbb{R}$ is a heterogeneous coefficient with multiple scales and of high contrast. Also, it satisfies $0 < \kappa_{min} \leq \kappa(x) \leq \kappa_{max}$, where $\kappa_{max}$ is large.

Denote $V := H(\text{div};\Omega)$, $Q := L^2(\Omega)$ and $V_0 := \{ v \in V : v \cdot \mathbf{n} = 0 \text{ on } \partial \Omega \}$. To solve the problem, we consider the following variational system: find $u\in V_0$ and $p\in Q$ such that 
\begin{eqnarray}
    a(u,v) - b(v,p) = &0& \quad \forall v\in V_0,  \label{eqn:first}\\
    b(u,q) = &(f,q)& \quad \forall q \in Q, \label{eqn:second}
\end{eqnarray}
where the bilinear forms $a(\cdot,\cdot)$ and $b(\cdot,\cdot)$ are defined as follows:
$$a(v,w) := \int_\Omega \kappa^{-1}v \cdot w \diff{x} \quad \text{and} \quad b(w,q) := \int_\Omega q\ \divg{w} \diff{x}.$$
We remark that \eqref{eqn:first}-\eqref{eqn:second} are solved together with the condition 
$\int_\Omega p \diff{x} = 0$. Note that the following inf-sup condition holds: for all $q\in Q$ with $\int_\Omega q \diff{x} = 0$, there is a constant $c>0$ which is independent to $\kappa_{max}$ such that
\begin{eqnarray}
\norm{q}_{L^2(\Omega)} \leq c \sup_{v\in V_0} \frac{b(v,q)}{\norm{v}_{H(\text{div}; \Omega)}}. \label{eqn:inf-sup}
\end{eqnarray} 
Next, we introduce the notions of fine and coarse grids. Let $\mathcal{T}^H$ be a conforming partition of the computational domain $\Omega$ with mesh size $H>0$. We refer to this partition as the coarse grid. Subordinate to the coarse grid, define the fine grid partition (with mesh size $h \ll H$), denoted by $\mathcal{T}^h$, by refining each coarse element into a connected union of fine grid blocks. We assume the above refinement is performed such that $\mathcal{T}^h$ is a conforming partition of $\Omega$.
Let $N$ be the number of coarse elements and $N_c$ be the number of interior coarse grid nodes of $\mathcal{T}^H$. 
The basis functions used for solving the problem are constructed based on the coarse grid. In the next section, we \reviii{will detail} the constructions of the basis functions for velocity and pressure.

\section{The multiscale method}
\label{sec:multiscale}
In this section, we present the framework of \reviii{CEM-GMsFEM}. The method consists of two general steps. First, we construct a multiscale space for approximating pressure. Next, we use the \reviii{pressure space} to construct another multiscale space for the velocity. \reviii{We remark} that the basis functions for pressure are local. That is, the support of each pressure basis is a coarse element. \reviii{For} each pressure basis function, we construct a corresponding velocity basis function \reviii{supported in} an oversampled region containing the support of the pressure basis function. The oversampled region is obtained by enlarging a coarse element by several coarse grid layers. This localized feature of the velocity basis function is the key to the proposed method.

\subsection{Construction of pressure basis functions}
We present the construction of the pressure basis functions for the mixed formulation. For each coarse element $K_i$, we construct a set of auxiliary multiscale basis functions using a specific spectral problem. For a set $S \subset \Omega$, define $Q(S) := L^2(S)$ and $V_0(S) := \{ v \in H(\text{div};\Omega) : v \cdot n = 0 \text{ on } \partial S \}$. Next, we define the required spectral problem. For each coarse element $K_i$, the spectral problem is to find $(\phi_j^{(i)}, p_j^{(i)}) \in V_0(K_i) \times Q(K_i)$ and the eigenvalue $\lambda_j^{(i)}\in \mathbb{R}$ such that
\begin{eqnarray}
    a(\phi_j^{(i)}, v) - b(v,p_j^{(i)} ) = & 0 & \quad \forall v \in V_0(K_i), \label{eqn:sp_pressure_1} \\
    b(\phi_j^{(i)}, q ) = & \lambda_j^{(i)} s_i(p_j^{(i)}, q)& \quad \forall q \in Q(K_i), \label{eqn:sp_pressure_2}
\end{eqnarray}
where $s_i$ is defined to be
$$ s_i(p,q) := \int_{K_i} \tilde{\kappa} pq \diff{x}, \quad \tilde{\kappa} := \kappa \sum_{\reviii{j=1}}^{N_c} \abs{\nabla \chi_j}^2$$
and $\{\chi_j \}_{j=1}^{N_c}$ is the set of standard multiscale basis functions satisfying the partition of unity property. Remark that one can also use other types of partition of unity functions. Assume that $s_i(p_j^{(i)},p_j^{(i)}) = 1$ and that we arrange the eigenvalues obtained from \eqref{eqn:sp_pressure_1}-\eqref{eqn:sp_pressure_2} in nondecreasing order: $0 = \lambda_1^{(i)} \leq \lambda_2^{(i)} \leq \cdots$. After that, we pick the first $J_i$ eigenfunctions $\{p_j^{(i)}\}$ corresponding to the largest $J_i$ eigenvalues $\lambda_j^{(i)}$ to define the local auxiliary space $Q_{aux}(K_i)$ by $$Q_{aux}(K_i) := \text{span} \{ p_j^{(i)} : 1\leq j \leq J_i \}.$$
The global auxiliary space $Q_{aux}$ is defined as
$Q_{aux} := \bigoplus_{i=1}^N Q_{aux}(K_i)$. Note that the space $Q_{aux}$ will be used as the approximation space for the pressure $p$.

\subsection{Construction of the multiscale basis functions for velocity}
In this section, we present the construction of the velocity basis function. For each pressure basis function in $Q_{aux}$, we construct a corresponding velocity basis function, whose support is an oversampled region containing the support of the pressure basis function. Define the projection operator $\pi: Q \rightarrow Q_{aux}$ by 
$$\pi(q) := \sum_{i=1}^N \sum_{j=1}^{J_i} s_i(q,p_j^{(i)}) p_j^{(i)}, \quad \forall q\in Q.$$ Note that $\pi$ is \reviii{the projection} of $Q$ onto $Q_{aux}$ with respect to the inner product $s(p,q) := \sum_{i=1}^N s_i(p,q)$.
Let $p_j^{(i)} \in Q_{aux}$ be a given pressure basis function supported in $K_i$. Let $K_{i,\ell}$ be an oversampled region obtained by enlarging $\ell$ layers from $K_i$. Namely, 
$$ K_{i,0} := K_i, \quad K_{i,\ell} := \bigcup \{ K \in \mathcal{T}^H : K \cap \overline{K_{i,\ell-1}} \neq  \emptyset \}, \quad \ell = 1,2, \cdots .$$
For simplicity, we may denote this oversampled region as $K_i^+$. 
The multiscale velocity basis function $\psi_{j,ms}^{(i)} \in V_0(K_i^+)$ is constructed by solving the following problem: find $(\psi_{j,ms}^{(i)},q_{j,ms}^{(i)}) \in V_0(K_i^+) \times Q(K_i^+)$ such that 
\begin{eqnarray}
    a(\psi_{j,ms}^{(i)},v) - b(v,q_{j,ms}^{(i)}) = & 0 & \quad \forall v \in V_0(K_i^+), \label{eqn:velocity_1} \\
    s(\pi q_{j,ms}^{(i)},\pi q) + b(\psi_{j,ms}^{(i)},q) = & s(p_{j}^{(i)},q) & \quad \forall q \in Q(K_i^+). \label{eqn:velocity_2}
\end{eqnarray}
The multiscale space for velocity can be defined as $V_{ms} := \{ \psi_{j,ms}^{(i)} : 1 \leq j \leq J_i, 1\leq i \leq N \}$. We remark that the construction for velocity basis function is motivated by the unconstraint energy minimization problem (24) in \cite{chung2018constraint}. One may also define the local basis function $\psi_{j,ms}^{(i)}$ in the region $K_i^+$ by a localization property of the related global basis function $\psi_j^{(i)}$. The global basis function $\psi_j^{(i)}\in V_0$ is defined as the solution to the following problem: find $(\psi_j^{(i)}, q_j^{(i)}) \in V_0\times Q$ such that 
\begin{eqnarray}
    a(\psi_{j}^{(i)},v) - b(v,q_{j}^{(i)}) = & 0 & \quad \forall v \in V_0, \label{eqn:global_velocity_1} \\
    s(\pi q_{j}^{(i)},\pi q) + b(\psi_{j}^{(i)},q) = & s(p_j^{(i)},q) & \quad \forall q \in Q. \label{eqn:global_velocity_2}
\end{eqnarray}
The global multiscale space is defined as $V_{glo} := \text{span}\{ \psi_j^{(i)}: 1 \leq j \leq J_i, 1 \leq i \leq N\}$.
Note that the system \eqref{eqn:global_velocity_1}-\eqref{eqn:global_velocity_2} defines a mapping $G: Q_{aux} \to V_{glo} \times Q$ in the following manner: given $p_{aux} \in Q_{aux}$, the image $G(p_{aux}) = (\psi,r) \in V_{glo} \times Q$ is defined as the solution to the following system: 
\begin{eqnarray}
    a(\psi,v) - b(v,r) = & 0 & \quad \forall v \in V_0, \label{eqn:global_velocity_3} \\
    s(\pi r,\pi q) + b(\psi,q) = & s(p_{aux},q) & \quad \forall q \in Q. \label{eqn:global_velocity_4}
\end{eqnarray}

\begin{remark}
\reviii{In practice, we may use a fixed number of oversampling layers $\ell \in \mathbb{N}^+$ for each $K_i^+$ to form the multiscale basis functions for velocity. }
\end{remark}

\begin{remark}
The mapping $G_1: p_{aux} \mapsto \psi$ is surjective and the operator $G$ is continuous. In particular, for any $p_{aux} \in Q_{aux}$, there exists a constant $C>0$ such that 
\begin{eqnarray}
	\Big( \norm{\psi}_a^2 + \norm{\pi r}_s^2 \Big)^{\frac{1}{2}} \leq C \norm{p_{aux}}_s. \label{eqn:G_con}
\end{eqnarray}
Taking $v = \psi$ in \eqref{eqn:global_velocity_3} and $q = r$ in \eqref{eqn:global_velocity_4} and summing over the equations, one obtains 
\begin{eqnarray*}
	\norm{\psi}_a^2 + \norm{\pi r}_s^2 & = & s(p_{aux}, r) \leq \norm{\tilde \kappa p_{aux}}_{L^2(\Omega)} \norm{r}_{L^2(\Omega)} \\
	& \leq & C_0 B \norm{p_{aux}}_s \norm{\psi}_a \leq  C_0 B \norm{p_{aux}}_s \Big( \norm{\psi}_a^2 + \norm{\pi r}_s^2 \Big)^{\frac{1}{2}},
\end{eqnarray*}
where $C_0$ is the constant in the inf-sup condition \eqref{eqn:inf-sup} and $B = \max_{x\in \Omega} \{\kappa(x) \}$.
\end{remark}

\subsection{The method}
The multiscale solution $(u_{ms}, p_{ms}) \in V_{ms} \times Q_{aux}$ is obtained by solving the following variational system:
\begin{eqnarray}
    a(u_{ms}, v) - b(v,p_{ms}) = & 0 & \quad \forall v \in V_{ms},\label{eqn:var_1} \\
    b(u_{ms},q) = & (f,q) & \quad \forall q \in Q_{aux}. \label{eqn:var_2}
\end{eqnarray}

To analyze the method, we define the norms for the multiscale spaces. 
Given a subset $D \subset \Omega$, for any $v\in V$ and $q\in Q$, define the norms $\norm{\cdot}_{a(D)}$, $\norm{\cdot}_{s(D)}$ and $\norm{\cdot}_{V(D)}$ by
\begin{eqnarray*}
    \norm{v}_{a(D)} & := & \bigg(\int_D \kappa^{-1}\abs{v}^2 \diff{x} \bigg)^{1/2}, \quad  \norm{q}_{s(D)} := \bigg(\int_D \tilde{\kappa} \abs{q}^2 \diff{x} \bigg)^{1/2}, \\
    \norm{v}_{V(D)} & := & \bigg(\int_D \tilde{\kappa}^{-1} \abs{\nabla \cdot v}^2 \diff{x} + \int_D \kappa^{-1} \abs{v}^2 \diff{x} \bigg)^{1/2}.
\end{eqnarray*}
For the case $D = \Omega$, we simply denote the norms as $\norm{\cdot}_a$, $\norm{\cdot}_s$ and $\norm{\cdot}_V$.

\section{Online adaptive enrichment}
\label{sec:online}
In this section, we will introduce an enrichment algorithm requiring the construction of new basis functions based on a given multiscale approximation. These functions constructed in this manner are called online basis functions as they are built in the online stage of computations.

\subsection{Construction of the online basis function}
To begin, we first define a pair of residual functionals. Let $(u_{ms}, p_{ms})$ be the current multiscale solution to the system \eqref{eqn:var_1}-\eqref{eqn:var_2} and $\{ \chi_i \}_{i=1}^{N_c}$ be a set of multiscale partition of unity corresponding to the set of coarse neighborhoods $\{\omega_i \}_{i=1}^{N_c}$, where 
$$\omega_i := \bigcup \{ K_j\in \mathcal{T}^H: x_i \in \overline{K_j} \} $$
and $x_i$ is a coarse node. Define the local residuals $\mathcal{R}_i: V \to \mathbb{R}$ and $r_i : Q \to \mathbb{R}$ as follows: for any $v\in V$ and $q\in Q$
$$ \mathcal{R}_i (v) := a(u_{ms},\chi_i v) - b_i(v,p_{ms}) \quad \text{and} \quad r_i(q) := (f,\chi_i q) - b(u_{ms},\chi_i q),$$
where $b_i(\cdot,\cdot)$ is the restriction of $b(\cdot,\cdot)$ in the neighborhood $\omega_i$. For $\tilde \ell \in \mathbb{N}$, define $\omega_{i,\tilde \ell}$ as follows: 
$$ \omega_{i,0} := \omega_i, \quad \omega_{i,\tilde \ell} := \bigcup \{ K \in \mathcal{T}^H: K \cap \overline{\omega_{i,\tilde \ell - 1}} \neq \emptyset \}, \quad \tilde \ell = 1,2,\cdots .$$ 
We denote $\omega_{i,\tilde \ell} := \omega_i^+$ for short and write $V(\omega_i^+) := H(\text{div}; \omega_i^+)$, $Q(\omega_i^+) = L^2(\omega_i^+)$. The construction of the online basis function is motivated by the local residuals $\mathcal{R}_i$ and $r_i$. One may look for the online basis functions $(\beta_{on}^{(i)},q_{on}^{(i)} ) \in V(\omega_i^+) \times Q(\omega_i^+)$ (both supported in $\omega_i^+$) such that
\begin{eqnarray}
	a(\beta_{on}^{(i)},v) - b(v,q_{on}^{(i)}) = & \mathcal{R}_i(v) & \quad \forall v \in V(\omega_i^+), \label{eqn:new_online_local_v} \\
	s(\pi q_{on}^{(i)}, \pi q) + b(\beta_{on}^{(i)}, q ) = & r_i(q) & \quad \forall q \in Q(\omega_i^+).
\label{eqn:new_online_local_p}
\end{eqnarray}
We remark that the above online basis functions are obtained in the local (oversampled) region $\omega_i^+$ and this is the result of a localization of the corresponding global online basis functions $(\beta_{glo}^{(i)}, q_{glo}^{(i)}) \in V_0 \times Q$ defined by 
\begin{eqnarray}
	a(\beta_{glo}^{(i)},v) - b(v,q_{glo}^{(i)}) = & \mathcal{R}_i(v) & \quad \forall v \in V, \label{eqn:new_online_global_v} \\
	s(\pi q_{glo}^{(i)}, \pi q) + b(\beta_{glo}^{(i)}, q ) = & r_i(q) & \quad \forall q \in Q.
\label{eqn:new_online_global_p}
\end{eqnarray}
See Figure \ref{fig:online_basis} for an illustration of such a local online basis function for velocity.
\begin{figure}[h!]
\centering
\mbox{
\subfigure[First component]{
\includegraphics[width = 2.6in]{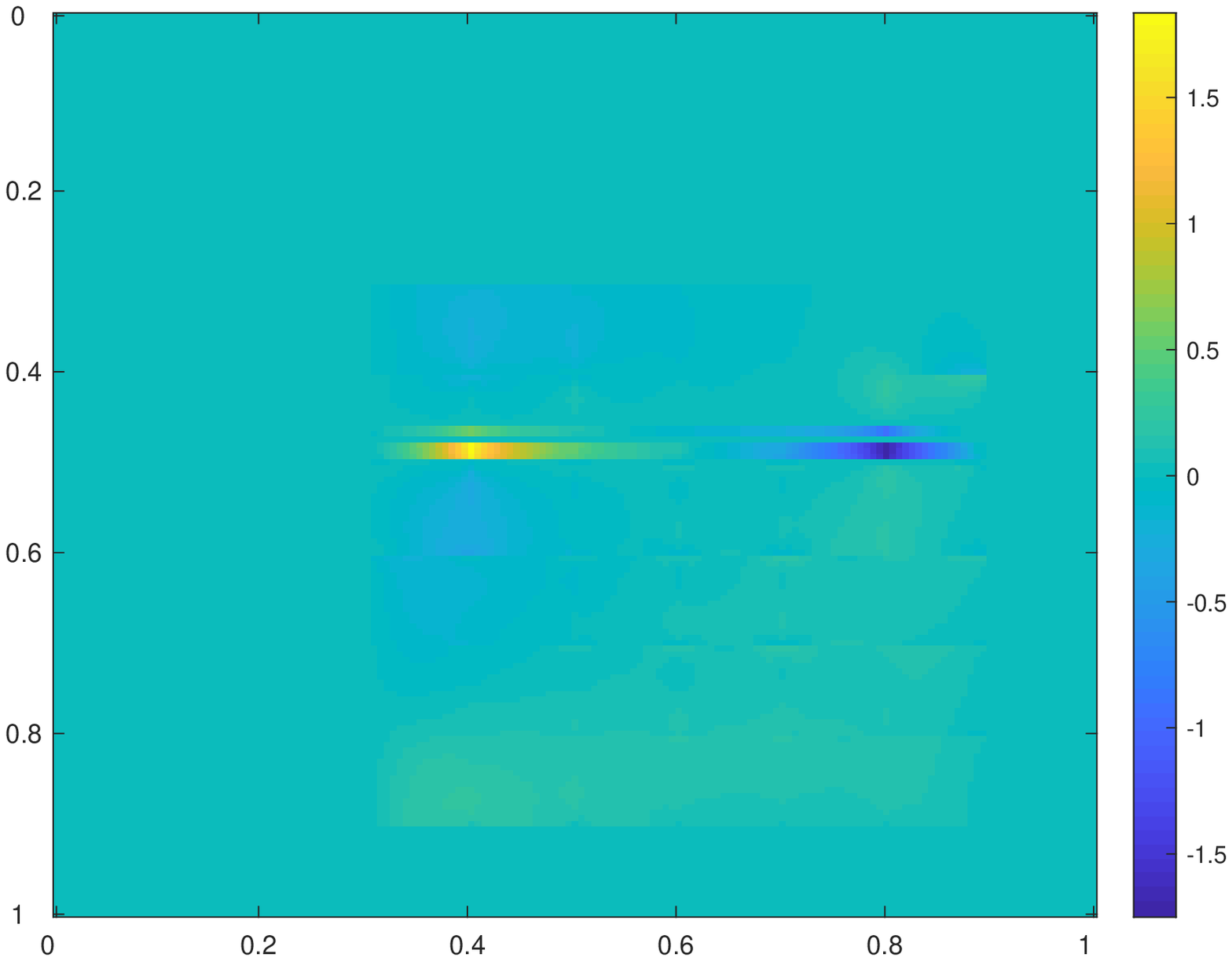}
\label{fig:online_basis_vx}}
\subfigure[Second component]{
\includegraphics[width = 2.6in]{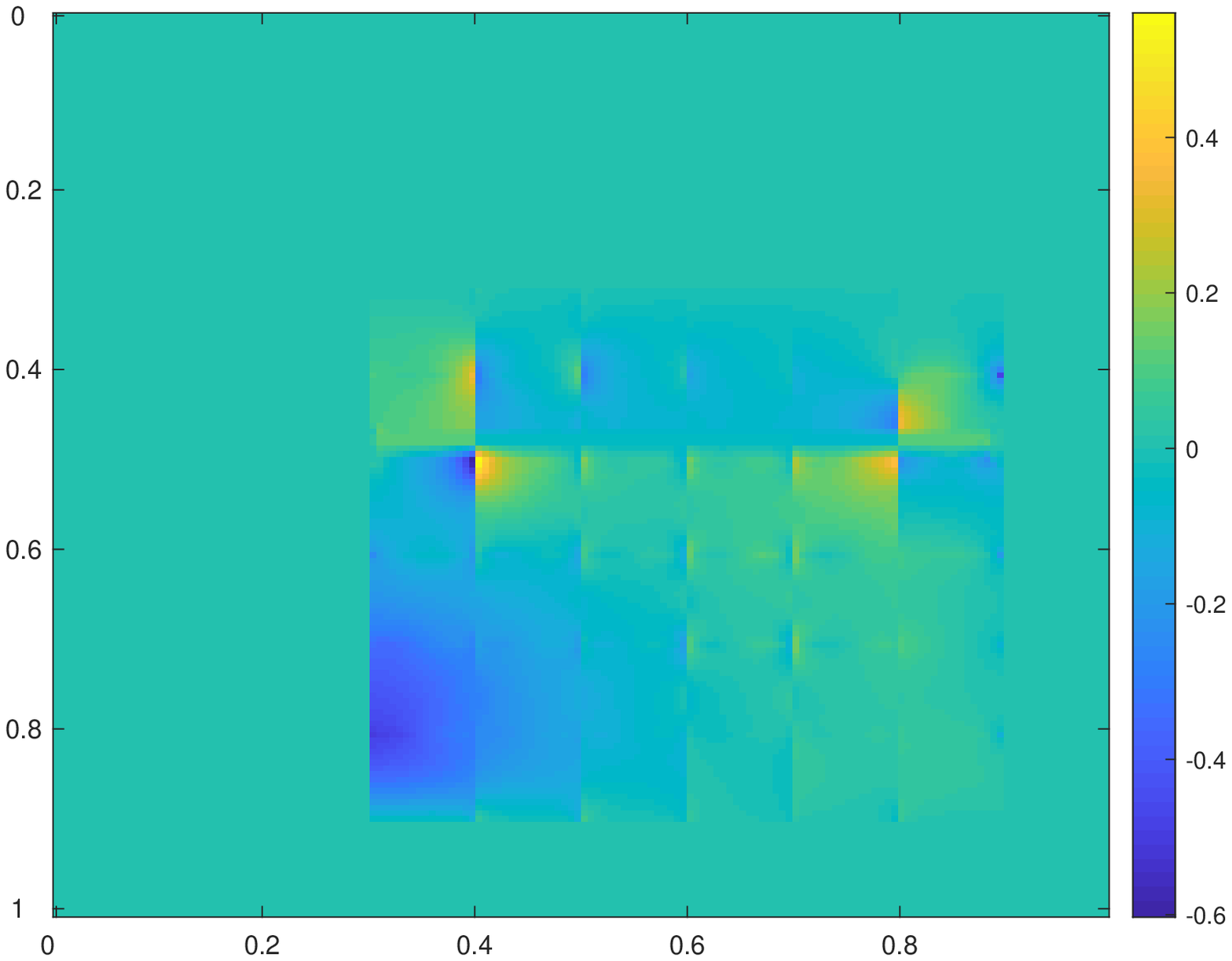}
\label{fig:online_basis_vy}}
}
\label{fig:online_basis}
\caption{An online basis function with $\tilde \ell = 2$.}
\end{figure}

\subsection{Online adaptive algorithm}
In this section, we describe the proposed online adaptive algorithm. 
First, we define the operator norms of the local residuals $\mathcal{R}_i$ and $r_i$ corresponding to $\omega_i \subset \Omega$ as follows: 
$$ \norm{\mathcal{R}_i}_{a^*} := \sup_{v \in V} \frac{\abs{\mathcal{R}_i (v)}}{\norm{v}_{a(\omega_i)}} \quad \text{and} \quad \norm{r_i}_{s^*} := \sup_{q \in Q} \frac{\abs{r_i(q)}}{\norm{q}_{s(\omega_i)}}.$$
In the following, we denote 
$$\mathcal{R}_i^m(v) := a(u_{ms}^m,\chi_i v) - b_i(v,p_{ms}^m) \quad \text{and} \quad r_i^m(q) := (f, \chi_i q) - b(u_{ms}^m, \chi_i q),$$
where $(u_{ms}^m, p_{ms}^m)$ is the multiscale solution at iteration level $m \in \mathbb{N}$. 
We detail the algorithm as follows. 
The index $m \in \mathbb{N}$ represents the level of online enrichment. Set $m = 0$ and initially define $V_{ms}^m := V_{ms}$. Choose a tolerance $\texttt{tol} \in \mathbb{R}_+$ and a parameter $\theta$ such that $0< \theta \leq 1$. For each $m \in \mathbb{N}$, assume that $V_{ms}^m$ is given. Go to Step 1 below.
\begin{itemize}
    \item [{\bf Step 1.}] Solve \eqref{eqn:var_1}-\eqref{eqn:var_2} over the multiscale spaces $V_{ms}^m \times Q_{aux}$ to obtain the current approximation $(u_{ms}^m, p_{ms}^m) \in V_{ms}^m \times Q_{aux}$.
    
    \item [{\bf Step 2.}] For each $i= 1,\cdots,N_c$, compute the norms of the local residuals to obtain $\eta_{i,m}^2 :=  \norm{\mathcal{R}_i^m}_{a^*}^2 + \norm{r_i^m}_{Q^*}^2$. Rearrange the indices of $\eta_{i,m}$ such that $\eta_{1,m} \geq \eta_{2,m} \geq \cdots \geq \eta_{N_c,m}$. Choose the smallest integer $k\in \mathbb{N}$ such that 
    $$ \sum_{i=1}^k \eta_{i,m}^2 \geq \theta \sum_{i=1}^{N_c} \eta_{i,m}^2.$$
    
    \item [{\bf Step 3.}] For each $i = 1,\cdots,k$, solve \eqref{eqn:new_online_local_v}-\eqref{eqn:new_online_local_p} over $\omega_i^+$ to obtain the online basis functions $(\beta_{on}^{(i)},q_{on}^{(i)})$. Enrich the mutlsicale space by letting
    $$ V_{ms}^{m+1} := V_{ms}^m \oplus \text{span}\{\beta_{on}^{(i)} : i=1,\cdots,k \}.$$
    
    \item [{\bf Step 4.}] 
    If $\sum_{i=1}^{N_c} \eta_{i,m}^2 \leq \texttt{tol}$ or there is a certain number of basis functions in $V_{ms}^{m+1}$, then {\bf Stop}. Otherwise, go back to {\bf Step 1} and set $m \leftarrow m+1$.    
\end{itemize}

\begin{remark}
\revii{Besides solving the multiscale approximations in Steps 1 and 4 above, one requires to compute the local residuals to obtain the error indicators $\eta_{i,m}$ in each iteration of the online enrichment. After that, one may adaptively select the local patches $\omega_i$, where the indicators are large, and construct the corresponding online basis functions in the enlarged regions $\omega_i^+$. Figure \ref{fig:flow_chart} depicts the procedure of the proposed online adaptive algorithm. 
}
\end{remark}

\begin{remark}
\reviii{Similar to the offline construction for velocity bases, we consider a uniform number of oversampling layers $\tilde \ell$ during the whole online procedure in order to simplify the implementation of the algorithm. }
\end{remark}

\tikzstyle{decision} = [diamond, draw, fill=blue!20, 
    text width=5.5em, text centered, node distance=3cm, inner sep=0pt]
\tikzstyle{block} = [rectangle, draw, fill=blue!20, 
    text width=7.5em, text centered, rounded corners, minimum height=4.2em]
\tikzstyle{line} = [draw, -latex']
\tikzstyle{cloud} = [draw, ellipse,fill=red!20, node distance=3cm,
    minimum height=2em]
\tikzstyle{arrow} = [->,>=stealth]    

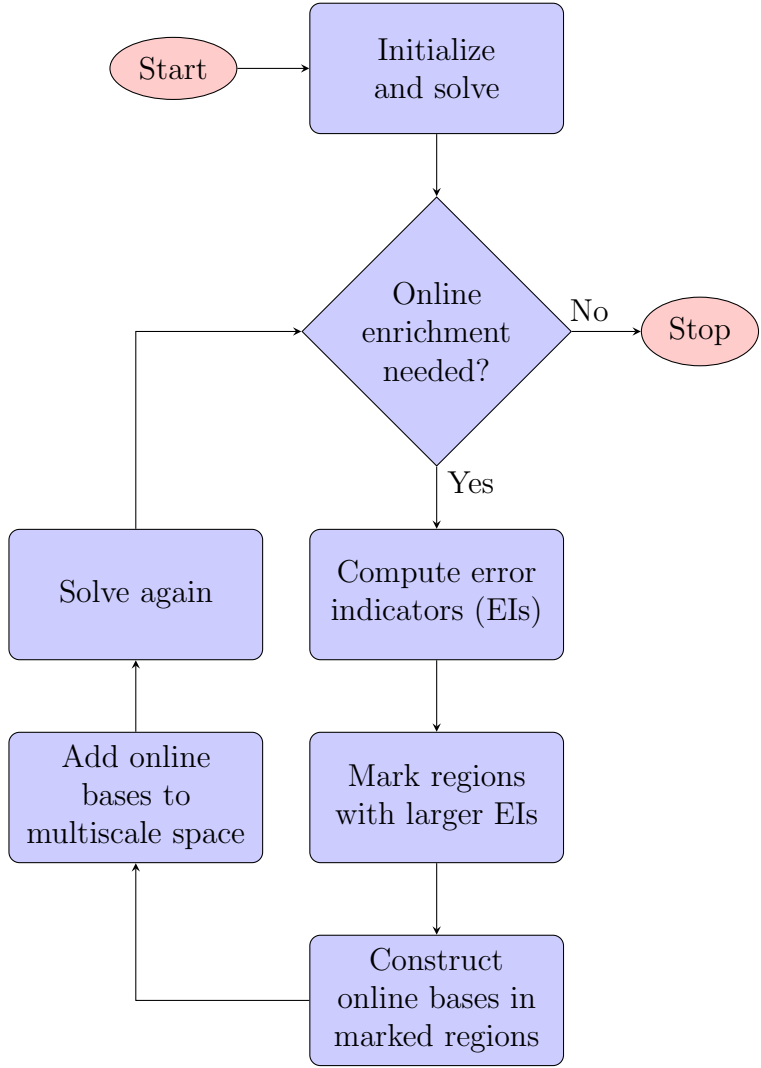
\begin{figure}[h!]
\centering
\begin{tikzpicture}[node distance = 2cm, auto]
    \node [cloud] (start) {Start};
    \node [block, right of=start, node distance = 3.5cm] (fs) {Initialize and solve};
    \node [decision, below of=fs, node distance = 3.5cm] (enrich) {Online enrichment needed?};
    \node [block, below of=enrich, node distance=3.5cm] (compute) {Compute error indicators (EIs)};
    \node [block, below of=compute, node distance=2.7cm] (sort) {Mark regions with larger EIs};
    \node [block, below of=sort, node distance=2.7cm] (construct) {Construct online bases in marked regions};
    \node [block, left of=sort, node distance=4cm] (add) {Add online bases to multiscale space};
    \node [block, left of=compute, node distance=4cm] (solve) {Solve again};
    \node [cloud, right of=enrich, node distance=3.5cm] (stop) {Stop};
    \draw [arrow] (start) -- (fs);
    \draw [arrow] (fs) -- (enrich);
    \draw [arrow] (enrich) -- node [near start] {Yes} (compute);
    \draw [arrow] (compute) -- (sort); 
    \draw [arrow] (sort) -- (construct);
    \draw [arrow] (construct) -| (add);
    \draw [arrow] (add) -- (solve);
    \draw [arrow] (solve) |- (enrich);
    \draw [arrow] (enrich) -- node [near start] {No} (stop);
\end{tikzpicture}
\caption{A graphical description of the proposed online enrichment strategy.} 
\label{fig:flow_chart}
\end{figure}

\subsection{Analysis}
In this section, we provide the convergence analysis of the proposed method. We denote $a \lesssim b$ if there is a generic constant $C>0$ such that $a \leq Cb$. Next, we recall some useful theoretical results from \cite{chung2018constraintmixed}.
\begin{lemma}[Lemma 1 in \cite{chung2018constraintmixed}]
\label{lem:1_mixed}
For each $p_{aux} \in Q_{aux}$ with $s(p_{aux},1) = 0 $, there is a unique $u \in V_{glo}$ such that $(u,p) \in V_0 \times Q$ (with $\int_\Omega p \diff{x} = 0$) is the solution of the following system: 
\begin{eqnarray*}
	a(u,v) - b(v,p) = & 0 & \quad \forall v \in V_0, \\
	b(u,q) = & s(p_{aux},q)& \quad \forall q \in Q.
\end{eqnarray*}
\end{lemma}
The following lemma motivates the local multiscale basis functions defined in \eqref{eqn:velocity_1}-\eqref{eqn:velocity_2}, saying that the global basis functions defined in \eqref{eqn:global_velocity_1}-\eqref{eqn:global_velocity_2} have an exponential decay outside an oversampled region. In the following, we denote $E$ as the constant of exponential decay 
$$E = C(1+\Lambda^{-1})\Big(1+C^{-1}(1+\Lambda^{-1})^{-\frac{1}{2}})\Big)^{1-\ell}, \quad \text{where }\Lambda := \min_{1\leq i \leq N} \lambda_{J_i +1}^{(i)}.$$ 

\begin{lemma}[Lemma 7 in \cite{chung2018constraintmixed}] 
\label{lem:local}
Let $(\psi_j^{(i)}, q_j^{(i)})$ be the solution of \eqref{eqn:global_velocity_1}-\eqref{eqn:global_velocity_2} and $(\psi_{j,ms}^{(i)}, q_{j,ms}^{(i)})$ be the solution \eqref{eqn:velocity_1}-\eqref{eqn:velocity_2}. For $K_i^+ = K_{i,\ell}$ with $\ell \geq 2$, we have 
$$ \norm{\psi_j^{(i)} - \psi_{j,ms}^{(i)}}_V^2 + \norm{q_j^{(i)} - q_{j,ms}^{(i)}}_s^2 \leq E\norm{p_j^{(i)}}_s^2.$$
\end{lemma}
Next, the following lemma is needed in the analysis of the main result. The proof of this lemma makes use of the cutoff function used in \cite{chung2018constraint} and one may see \cite{chung2018fast} for more details about this lemma.
\begin{lemma}[cf. Lemma 3 in \cite{chung2018fast}]
\label{lem:add_norm_sp}
Assume the same conditions in Lemma \ref{lem:local} hold. For any $\{ d_j^{(i)} \} \subset \mathbb{R}$ and $\widetilde C = C(1+\Lambda^{-1})(\ell+1)^d$, we have
\begin{eqnarray*}
\Norm{\sum_{i=1}^N \sum_{j=1}^{J_i} d_j^{(i)} (\psi_j^{(i)} - \psi_{j,ms}^{(i)})}_V^2 &+& \Norm{\sum_{i=1}^N \sum_{j=1}^{J_i} d_j^{(i)} \pi(q_j^{(i)} - q_{j,ms}^{(i)})}_s^2  \\
&\leq& \widetilde C \sum_{i=1}^N \bigg( \Norm{ \sum_{j=1}^{J_i} d_j^{(i)} (\psi_j^{(i)} - \psi_{j,ms}^{(i)})}_V^2  
+ \Norm{\sum_{j=1}^{J_i} d_j^{(i)} \pi(q_j^{(i)} - q_{j,ms}^{(i)} )}_s^2 \bigg).
\end{eqnarray*}
\end{lemma}
Note that the basis functions constructed during the online stage are defined the same as the multiscale one. The same localization result holds for the online basis functions and the proof of the following lemma is the same as that for Lemmas \ref{lem:local} and \ref{lem:add_norm_sp}, and is therefore omitted.
\begin{lemma}
\label{lem:online_local}
	Assume that $\omega_i^+$ is obtained from $\omega_i$ by enlarging $\tilde \ell$ coarse grid layers with $\tilde \ell \geq 2$. Let $(\beta_{on}^{(i)}, q_{on}^{(i)})$ be the local online basis functions in \eqref{eqn:new_online_local_v}-\eqref{eqn:new_online_local_p} and $(\beta_{glo}^{(i)}, q_{glo}^{(i)})$ be the global one in \eqref{eqn:new_online_global_v}-\eqref{eqn:new_online_global_p}. Then, we have 
	$$\norm{\beta_{glo}^{(i)}- \beta_{on}^{(i)}}_V^2 + \norm{q_{glo}^{(i)}-q_{on}^{(i)}}_s^2 \leq \tilde E \big( \norm{\beta_{glo}^{(i)}}_a^2 + \norm{\pi q_{glo}^{(i)}}_s^2 \big),$$
where $\tilde E = C(1+\Lambda^{-1})\Big(1+C^{-1}(1+\Lambda^{-1})^{-\frac{1}{2}})\Big)^{1-\tilde \ell}$.
Furthermore, we have 
\begin{eqnarray*}
\Norm{\sum_{i=1}^{N_c} (\beta_{glo}^{(i)} - \beta_{on}^{(i)})}_V^2 &+& \Norm{\sum_{i=1}^{N_c}  \pi(q_{glo}^{(i)} - q_{on}^{(i)})}_s^2  \\
&\leq& \widetilde C \sum_{i=1}^{N_c} \big( \norm{  (\beta_{glo}^{(i)} - \beta_{on}^{(i)})}_V^2  
+ \norm{ \pi(q_{glo}^{(i)} - q_{on}^{(i)} )}_s^2 \big).
\end{eqnarray*}
\end{lemma}
\reviii{
\begin{lemma}[Lemma 8 in \cite{chung2018constraintmixed}]\label{lem:dis_inf_sup}
Assume that $\ell = O(\log(\kappa/H^2))$. For any $q \in Q_{aux}$ with $s(q,1) = 0$, there is $u \in V_{ms}$ such that 
$$\norm{q}_s \leq C_{ms} \frac{b(u,q)}{\norm{u}_V},$$
where $C_{ms}>0$ is a constant. 
\end{lemma}
}
The main result in the research reads as follows.
\begin{theorem}
Assume that $u \in V$ is the solution to \eqref{eqn:first}-\eqref{eqn:second} and for $m \in \mathbb{N}$, $u_{ms}^{m} \in V_{ms}^{m}$ is the approximated solution to \eqref{eqn:var_1}-\eqref{eqn:var_2}. Then, there are constants $C_0 = C_0(M,\ell, d, \Lambda)$, $C_1 = C_1(M,\tilde \ell,d,\Lambda)$ and $C_2 = C_2(M,\Lambda)$ such that 
$$ \norm{u - u_{ms}^{m+1}}_V^2 \leq \big(C_0 E + C_1 \tilde E + C_2 \tilde \theta \big) \reviii{C_{ms}}\norm{u-u_{ms}^m}_V^2,$$
where $M = \max_K n_K$ with $n_K$ being the number of coarse nodes of the coarse element $K$ and $\tilde \theta$ is the chosen adaptive parameter satisfying $\tilde \theta = 1-\theta$.
\end{theorem}
\begin{proof}
To simplify the analysis, we may assume the test function $v$ and $\sum_{i} \beta_{glo}^{(i)}$ to be divergence free. Recall the definition of the global online basis function corresponding to $\omega_i$: find $(\beta_{glo}^{(i)}, q_{glo}^{(i)}) \in V_0 \times Q$ such that 
\begin{eqnarray}
	a(\beta_{glo}^{(i)},v) - b(v,q_{glo}^{(i)}) = & \mathcal{R}_i(v) & \quad \forall v \in V, \label{eqn:on_glo_v_div}\\ 
	s(\pi q_{glo}^{(i)}, \pi q) + b(\beta_{glo}^{(i)}, q ) = & r_i(q) & \quad \forall q \in Q.
\label{eqn:on_glo_p_div}
\end{eqnarray}
After summing over all the neighborhoods $\omega_i$ for $i= 1,\cdots, N_c$, one obtains
\begin{eqnarray}
	a\bigg(u_{ms}^m - u + \sum_{i=1}^{N_c} \beta_{glo}^{(i)},v\bigg) - b\bigg(v,\sum_{i=1}^{N_c} q_{glo}^{(i)}\bigg) = & 0 & \quad \forall v \in  V, \label{eqn:glo_online_v} \\
	b\bigg(u_{ms}^m - u + \sum_{i=1}^{N_c} \beta_{glo}^{(i)}, q \bigg) = & s\bigg(\pi \bigg (- \sum_{i=1}^{N_c} q_{glo}^{(i)}\bigg), q\bigg) & \quad \forall q \in Q. \label{eqn:glo_online_p}
\end{eqnarray}
By Lemma \ref{lem:1_mixed}, there exists a set of constants $\{ c_j^{(i)}: 1\leq j \leq J_i, ~ 1\leq i \leq N \}$ such that 
$$ \eta := u_{ms}^m - u + \sum_{i=1}^{N_c} \beta_{glo}^{(i)} = \sum_{i=1}^N \sum_{j=1}^{J_i} c_j^{(i)} \psi_j^{(i)}.$$ 
Denote $ \xi := \sum_{i,j} c_j^{(i)} q_j^{(i)}$ and $p_{aux} := \sum_{i,j} c_j^{(i)} p_j^{(i)}$. Then, we have 
\begin{eqnarray*}
	a(\eta,v) - b(v,\xi) = & 0 & \quad v \in V, \\
	s(\pi \xi, \pi q) + b(\eta,q) = & s(p_{aux}, q) & \quad q \in Q.
\end{eqnarray*}
Note that $G(p_{aux}) = (\eta,\xi)$. Next, we estimate the error of localization, namely, the term $\sum_{i,j} c_j^{(i)} (\psi_j^{(i)}- \psi_{j,ms}^{(i)})$. By Lemmas \ref{lem:local} and \ref{lem:add_norm_sp}, we have 
\begin{eqnarray*}
	\Norm{\sum_{i=1}^N \sum_{j=1}^{J_i} c_j^{(i)} ( \psi_j^{(i)} - \psi_{j,ms}^{(i)})}_V^2 \leq \widetilde C(1+\Lambda^{-1}) (\ell +1)^d E \sum_{i=1}^N \sum_{j=1}^{J_i} (c_j^{(i)})^2.
\end{eqnarray*}
Applying the corollary of open mapping theorem \cite{ciarlet2013linear} to the mapping $G_1: Q_{aux} \to V_{glo}$, there is a constant $C_m >0$ such that 
$$ \norm{p_{aux}}_s^2 = \sum_{i=1}^N \sum_{j=1}^{J_i} (c_j^{(i)})^2 \leq C_m \norm{\eta}_a^2.$$
Taking $v = \sum_{i=1}^{N_c} \beta_{glo}^{(i)}$ in \eqref{eqn:glo_online_v}, $q = \pi \sum_{i=1}^{N_c} q_{glo}^{(i)}$ in \eqref{eqn:glo_online_p} and summing over these equations, we have 
$$ \Norm{\sum_{i=1}^{N_c} \beta_{glo}^{(i)}}_a^2 + \Norm{\sum_{i=1}^{N_c} \pi q_{glo}^{(i)}}_s^2 = a\bigg(u-u_{ms}^m, \sum_{i=1}^{N_c} \beta_{glo}^{(i)} \bigg) \implies  \Norm{\sum_{i=1}^{N_c} \beta_{glo}^{(i)}}_a \leq \norm{u - u_{ms}^m }_a.$$
Hence, we have 
\begin{equation}
\label{eqn:estimate_3}
\begin{split}
	\Norm{\sum_{i=1}^N \sum_{j=1}^{J_i} c_j^{(i)} ( \psi_j^{(i)} - \psi_{j,ms}^{(i)})}_V^2 & \leq  \widetilde C(1+\Lambda^{-1}) (\ell +1)^d E\sum_{i=1}^N \sum_{j=1}^{J_i} (c_j^{(i)})^2 \\
& \leq  \widetilde C C_m(1+\Lambda^{-1})(\ell +1)^d  E\norm{\eta}_a^2 \\
& \leq  2\widetilde C C_m(1+\Lambda^{-1})(\ell+1)^d E\bigg(\norm{u - u_{ms}^m}_a^2 + \Norm{\sum_{i=1}^{N_c} \beta_{glo}^{(i)}}_a^2 \bigg) \\
& \leq  4\widetilde C C_m(1+\Lambda^{-1})(\ell+1)^d E\norm{u-u_{ms}^m}_V^2.
\end{split}
\end{equation}
After that, we estimate the error of localization for the online basis functions.
Note that by Lemma \ref{lem:online_local} we have
$$ \norm{\beta_{glo}^{(i)} - \beta_{on}^{(i)}}_V^2 + \norm{\pi(q_{glo}^{(i)}- q_{on}^{(i)})}_s^2 \leq \tilde E \big( \norm{\beta_{glo}^{(i)}}_a^2 + \norm{\pi q_{glo}^{(i)}}_s^2 \big).$$
Taking $v = \beta_{glo}^{(i)}$ in \eqref{eqn:on_glo_v_div}, $ q = q_{glo}^{(i)}$ in \eqref{eqn:on_glo_p_div} and summing these two equations up, by Cauchy-Schwartz inequality one obtains 
\begin{eqnarray*}
 \norm{\beta_{glo}^{(i)}}_a^2 + \norm{\pi q_{glo}^{(i)}}_s^2 & = & a(u-u_{ms}^m, \chi_i \beta_{glo}^{(i)}) + b(u-u_{ms}^m, \chi_i q_{glo}^{(i)}) \\
 & \leq& \norm{u-u_{ms}^m}_{a(\omega_i)} \norm{\chi_i \beta_{glo}^{(i)}}_a + \norm{\tilde \kappa^{-1/2} \divg{u-u_{ms}^m}}_{L^2(\omega_i)} \norm{\chi_i q_{glo}^{(i)}}_s \\
 & \leq & \norm{u-u_{ms}^m}_{V(\omega_i)} \big( \norm{\chi_i \beta_{glo}^{(i)}}_a^2 + \norm{\chi_i q_{glo}^{(i)}}_s^2 \big)^{1/2} \\
 & \leq & (1+\Lambda^{-1}) \norm{u-u_{ms}}_{V(\omega_i)} \big( \norm{\beta_{glo}^{(i)}}_a^2 + \norm{q_{glo}^{(i)}}_s^2 \big)^{1/2}. 
 \end{eqnarray*}
It implies that 
\begin{equation}
\label{eqn:estimate_1}
\begin{split}
	\sum_{i=1}^{N_c} \big( \norm{\beta_{glo}^{(i)}- \beta_{on}^{(i)}}_V^2 + \norm{q_{glo}^{(i)} - q_{on}^{(i)}}_s^2 \big) & \leq  \tilde E(1+\Lambda^{-1}) \sum_{i=1}^{N_c} \norm{u-u_{ms}^m}_{V(\omega_i)}^2 \\
	& \leq  M\tilde E(1+\Lambda^{-1}) \norm{u-u_{ms}^m}_V^2.
\end{split}
\end{equation}
Finally, we may prove the required convergency. Let $\mathcal{I} = \{ 1, 2, \cdots, k \} \subset \{ 1,2,\cdots, N_c \}$ be the set of indices and we add online basis functions $\beta_{ms}^{(i)}$ for $i\in \mathcal{I}$ into the space $V_{ms}^m$ to form $V_{ms}^{m+1}$. Define the following function $w \in V_{ms}^{m+1}$: 
$$ w = u_{ms}^m - \sum_{i\in \mathcal{I}} \beta_{on}^{(i)} + \sum_{i=1}^N \sum_{j=1}^{J_i} c_j^{(i)} \psi_{j,ms}^{(i)}.$$
\reviii{Since the discrete inf-sup condition holds (Lemma \ref{lem:dis_inf_sup}), by the result of \cite[Theorem 12.5.17]{brenner2007mathematical}, we obtain}
\begin{eqnarray*}
	\norm{u-u_{ms}^{m+1}}_V^2 & \lesssim & \reviii{C_{ms}}\norm{u-w}_V^2 \\
	& = & \reviii{C_{ms}}\Norm{ \sum_{i \in \mathcal{I}} (\beta_{on}^{(i)} - \beta_{glo}^{(i)}) - \sum_{i \notin \mathcal{I}} \beta_{glo}^{(i)} + \sum_{i=1}^N \sum_{j=1}^{J_i} c_j^{(i)}( \psi_j^{(i)}- \psi_{j,ms}^{(i)})}_V^2 \\
	& \leq & 3\reviii{C_{ms}} \bigg( \underbrace{\Norm{\sum_{i \in \mathcal{I}} \beta_{on}^{(i)} - \beta_{glo}^{(i)}}_V^2}_{\text{\Large \ding{192}}} + \underbrace{\Norm{\sum_{i \notin \mathcal{I}} \beta_{glo}^{(i)}}_V^2}_{\text{\Large \ding{193}}} + \underbrace{\Norm{\sum_{i=1}^N \sum_{j=1}^{J_i} c_j^{(i)} (\psi_j^{(i)}- \psi_{j,ms}^{(i)})}_V^2}_{\text{\Large \ding{194}}} \bigg).
\end{eqnarray*}
By Lemma \ref{lem:online_local}, \eqref{eqn:estimate_3}, and \eqref{eqn:estimate_1} we have
\begin{eqnarray}
\label{eqn:estimate_1_fin}
\text{\Large \ding{192}} & \leq & \widetilde C M\tilde E(1+\Lambda^{-1}) \norm{u-u_{ms}^m}_V^2,\\
\label{eqn:estimate_3_fin}
\text{\Large \ding{194}} & \leq & 4 \widetilde C C_m (1+\Lambda^{-1})(\ell +1 )^d E\norm{u-u_{ms}^m}_V^2.
\end{eqnarray}
Next, we denote $\widetilde \beta = \sum_{i \notin \mathcal{I}} \beta_{glo}^{(i)}$ and $\widetilde q = \sum_{i \notin \mathcal{I}} q_{glo}^{(i)}$. By the definition of the online basis function, one may have 
\begin{eqnarray*}
	a(\widetilde \beta,v ) - b(v, \widetilde q) = & \sum_{i\notin \mathcal{I}} \mathcal{R}_i(v) & \quad \forall v \in \widehat V, \\
	s(\pi \widetilde q, \pi q) + b(\widetilde \beta,q) = & \sum_{i\notin \mathcal{I}} r_i(q) & \quad \forall q \in Q.
\end{eqnarray*}
Taking $v = \widetilde \beta$, $q = \widetilde q$ and adding two equations together, we have
\begin{eqnarray*}
	\norm{\widetilde \beta}_a^2 + \norm{\pi \widetilde q}_s^2 & = & \displaystyle{\sum_{i \notin \mathcal{I}} \mathcal{R}_i (\widetilde \beta) + r_i (\widetilde q)} \\
	& \leq & \displaystyle{\sum_{i \notin \mathcal{I}}} (\norm{\mathcal{R}_i}_{a^*}^2 + \norm{r_i}_{s^*}^2)^{1/2} (\norm{\widetilde \beta}_{a(\omega_i)}^2 + \norm{\widetilde q}_{s(\omega_i)}^2 )^{1/2} \\
	& \leq & (1+\Lambda^{-1})^{1/2} \displaystyle{\sum_{i \notin \mathcal{I}}} \eta_i (\norm{\widetilde \beta}_{a(\omega_i)}^2 + \norm{\pi \widetilde q}_{s(\omega_i)}^2 )^{1/2} \\
	& \leq & \sqrt{\tilde \theta(1+\Lambda^{-1})} \bigg( \displaystyle{\sum_{i=1}^{N_c}} \eta_i^2\bigg)^{1/2}
	\bigg( \displaystyle{\sum_{i=1}^{N_c} \norm{\widetilde \beta}_{a(\omega_i)}^2 + \norm{\pi \widetilde q}_{s(\omega_i)}^2  }\bigg)^{1/2}  \\
	& \leq & \sqrt{M\tilde \theta(1+\Lambda^{-1})} \bigg( \displaystyle{\sum_{i=1}^{N_c}} \eta_i^2\bigg)^{1/2} (\norm{\widetilde \beta}_{a}^2 + \norm{\pi \widetilde q}_{s}^2)^{1/2},
\end{eqnarray*}
where $\eta_i = (\norm{\mathcal{R}_i}_{a^*}^2 + \norm{r_i}_{s^*}^2)^{1/2}$ and $\tilde \theta = 1- \theta$. It implies that 
\begin{equation}
\label{eqn:estimate_2}
\begin{split}
	\norm{\widetilde \beta}_a^2 + \norm{\pi \widetilde q}_s^2 & \leq M\theta(1+\Lambda^{-1}) \sum_{i=1}^{N_c} \eta_i^2 
	 \leq M\theta(1+\Lambda^{-1}) \sum_{i=1}^{N_c} \norm{u - u_{ms}^m}_{V(\omega_i)}^2 \\
	& \leq M^2\tilde \theta(1+\Lambda^{-1}) \norm{u-u_{ms}^m}_V^2,
\end{split}
\end{equation}
where we use the fact that 
$$\abs{\mathcal{R}_i(v)} = \abs{a(u-u_{ms}^m,\chi_i v)} \leq \norm{u-u_{ms}^m}_{a(\omega_i)} \norm{v}_{a(\omega_i)} \quad \forall v \in V, $$
$$\abs{r_i(q)} = \abs{b(u-u_{ms}^m,\chi_i q)} \leq \norm{\tilde \kappa^{-1/2} \divg{u-u_{ms}^m}}_{L^2(\omega_i)} \norm{q}_{s(\omega_i)} \quad \forall q \in Q$$
$$ \implies \eta_i \leq \norm{u-u_{ms}^m}_{V(\omega_i)}.$$
Hence, the following estimate holds: 
\begin{equation}
\label{eqn:estimate_2_fin}
\text{\Large \ding{193}} \leq M^2\tilde \theta (1+\Lambda^{-1}) \norm{u-u_{ms}^m}_V^2.
\end{equation}
Combining \eqref{eqn:estimate_1_fin}, \eqref{eqn:estimate_3_fin}, and \eqref{eqn:estimate_2_fin}, we have 
\begin{equation}
\label{eqn:estimate_fin}
\norm{u- u_{ms}^{m+1}}_V^2 \leq (C_0 E +  C_1 \tilde E + C_2\tilde \theta ) \reviii{C_{ms}} \norm{u-u_{ms}^m}_V^2,
\end{equation}
where $C_0 = 12\widetilde C C_m(1+\Lambda^{-1})(\ell +1)^d$, $C_1 = 3\widetilde C M (1+\Lambda^{-1})$, and $C_2 = 3(1+\Lambda^{-1})M^2$.
This completes the proof.
\end{proof}

\begin{remark}
The quantity $\tilde \theta = 1- \theta$ is a user-defined parameter which can control the convergence rate of the online enrichment process. 
Note that the larger the parameter $\theta \in (0,1]$ used, the steeper the decay of error that occurs. 
This fact will be illustrated by our numerical experiments in the next section. On the other hand, 
the constant of exponential decay $\tilde E$ is defined in the online stage. One may choose a different number of layers in the construction of online basis functions to drive the error decay faster, even with a smaller number of layer $\ell$ in the offline stage.
\end{remark}

\section{Numerical experiment}
\label{sec:numerics}
In this section, we present some numerical results to show the efficiency of the proposed method. The computational domain is $\Omega = (0,1)^2$. We use a rectangular mesh for the partition of the domain dividing $\Omega$ into $T \times T$ equal coarse square blocks and further divide each coarse block into $n \times n$ equal square pieces. In other words, the fine mesh contains $Tn \times Tn$ fine rectangular elements with the mesh size $h = \frac{1}{Tn}$. The boundary condition is set to be $g = 0$. We test the performance by considering uniform enrichment ($\theta = 1$) and by using the online adaptive enrichment. In all the examples below, the term {\it energy error} refers to the following quantity:  
$$e_u := \frac{\norm{u-u_{ms}^m}_a}{ \norm{u}_a},$$ 
where $u$ is the reference solution solved on the fine mesh \revii{and $u_{ms}^m$ is the multiscale solution in the multiscale space $V_{ms}^m$. The index $m \in \mathbb{N}$ denotes the level of online iteration. In the examples below, we set the number of initial basis functions to be $J \equiv J_i = 3$.} \reviii{For each example below, the computational time in each online iteration is also reported in each table reporting the error quantity.} 
\begin{example}\label{exp:1}
	In this example, we set $T = 8$ and $n = 12$. The permeability field $\kappa$ used in this example is given in Figure \ref{fig:exp_1}. The source function $f$ is defined as follows: 
	$$ f(x) = \left \{ \begin{array}{cc}
	1 & x \in  (0,1/8)^2, \\
	-1 & x\in  (7/8,1)^2, \\
	0 & \text{otherwise},
	\end{array} \right .$$
	and thus the compatibility condition holds. The number of oversampling layers is $ \ell = 2$, $\tilde \ell = 2$. 
	\reviii{The computational time for offline construction is 5.9989 seconds.} 
	\revii{The profiles of the solutions are sketched in Figures \ref{fig:exp1_ref_sol} and \ref{fig:exp1_ms_sol}.} 
In Table \ref{tab:exp_1_theta_1}, we present the energy error for the case with uniform enrichment, that is, $\theta = 1$. One may observe a moderately fast convergence of the method.  In Table \ref{tab:exp_1_theta_0.1}, we present the energy error by using the online adaptive enrichment with $\theta = 0.1$. That is, only the basis functions related to the regions which account for the largest $10\%$ will be added in the online stage. Here, {\it DOF} stands for the dimension of the multiscale space $V_{ms}^m$. From the tables, we observe that smaller parameter $\theta$ leads to slower convergence. This confirms that the user-defined parameter is useful in controlling the convergence rate of the proposed adaptive method. 

Note that one may use a larger number of layers in online stage to further reduce the error of velocity. In Table \ref{tab:exp_1_ol4}, we present the results with offline number of layers $\ell = 2$ and set the online number of layers as $\tilde \ell = 4$. One may observe that the larger number of layers in the online stage leads to a faster decay in error reduction with less iterations.
	\begin{figure}[h!]	
	\centering
	\mbox{
	\subfigure[Example \ref{exp:1}.]{
	\centering
	\includegraphics[width=2.7in]{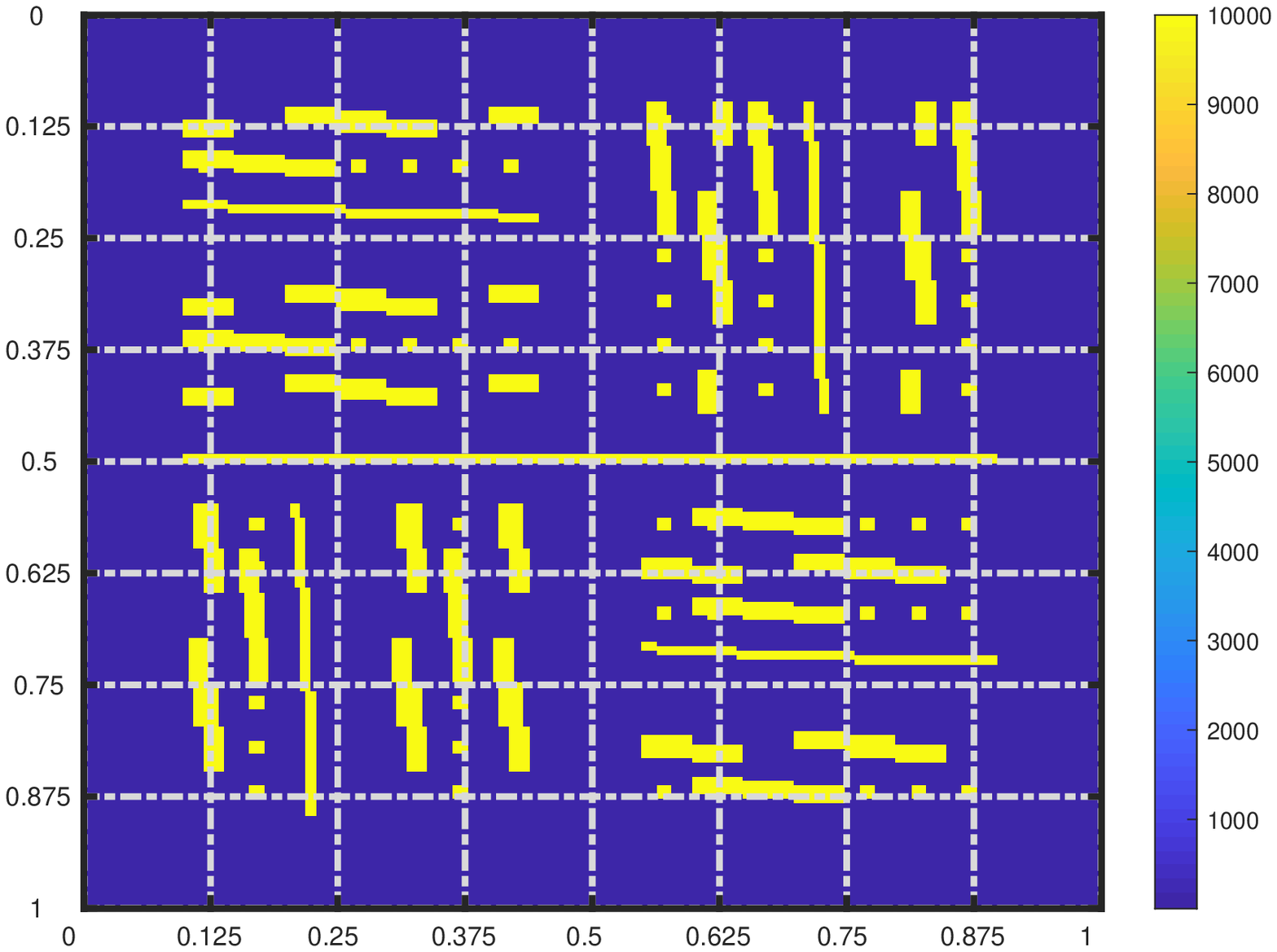}
	\label{fig:exp_1_kappa}}
	\subfigure[Example \ref{exp:2}.]{
	\centering
	\includegraphics[width=2.6in]{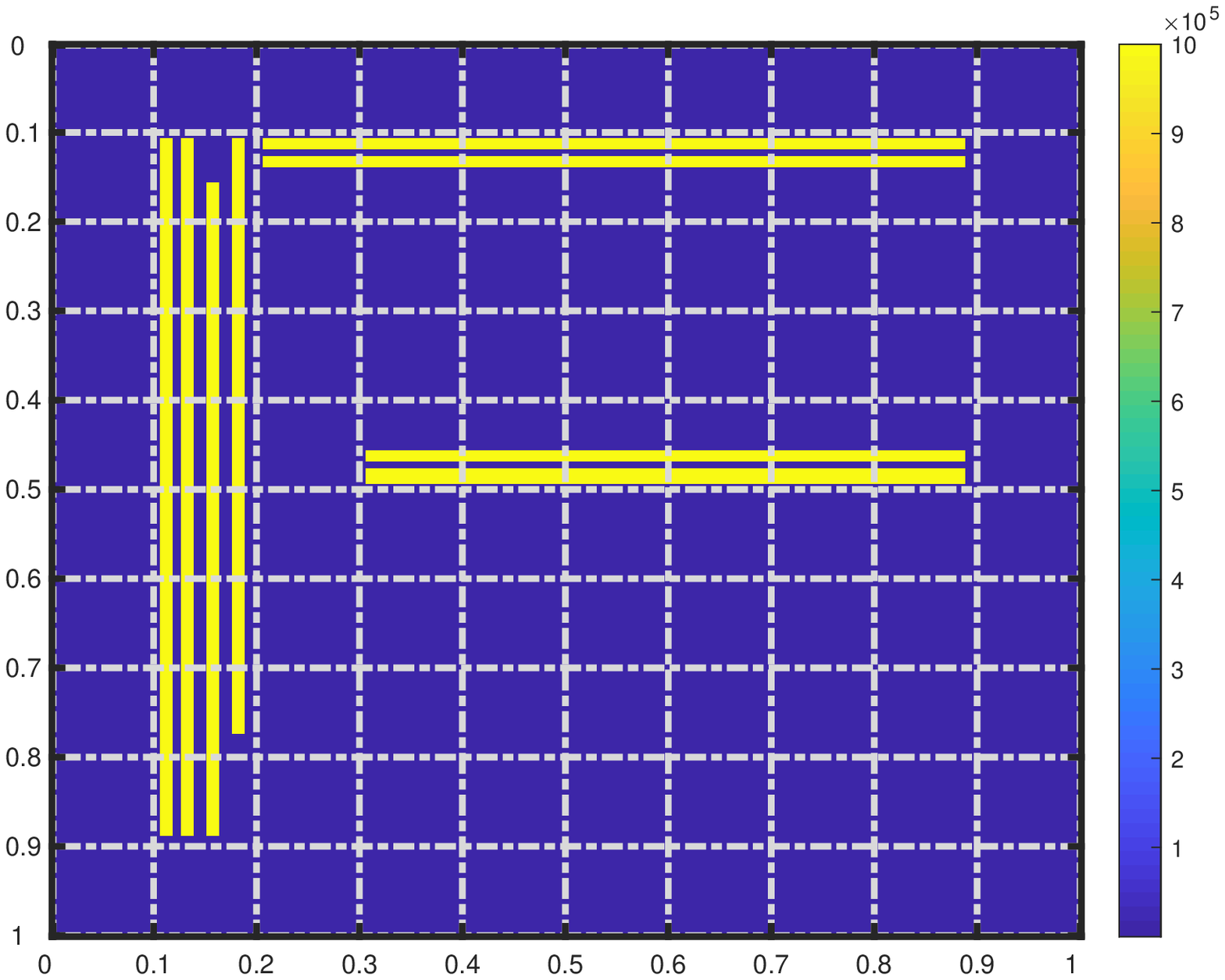}
	\label{fig:exp_2_kappa}}
	}
	\caption{The permeability field $\kappa$.}
	\label{fig:exp_1}
	\end{figure}
	
	\begin{figure}[h!]
	\centering
	\mbox{
	\subfigure[Pressure.]{
	\centering
	\includegraphics[width=2.6in]{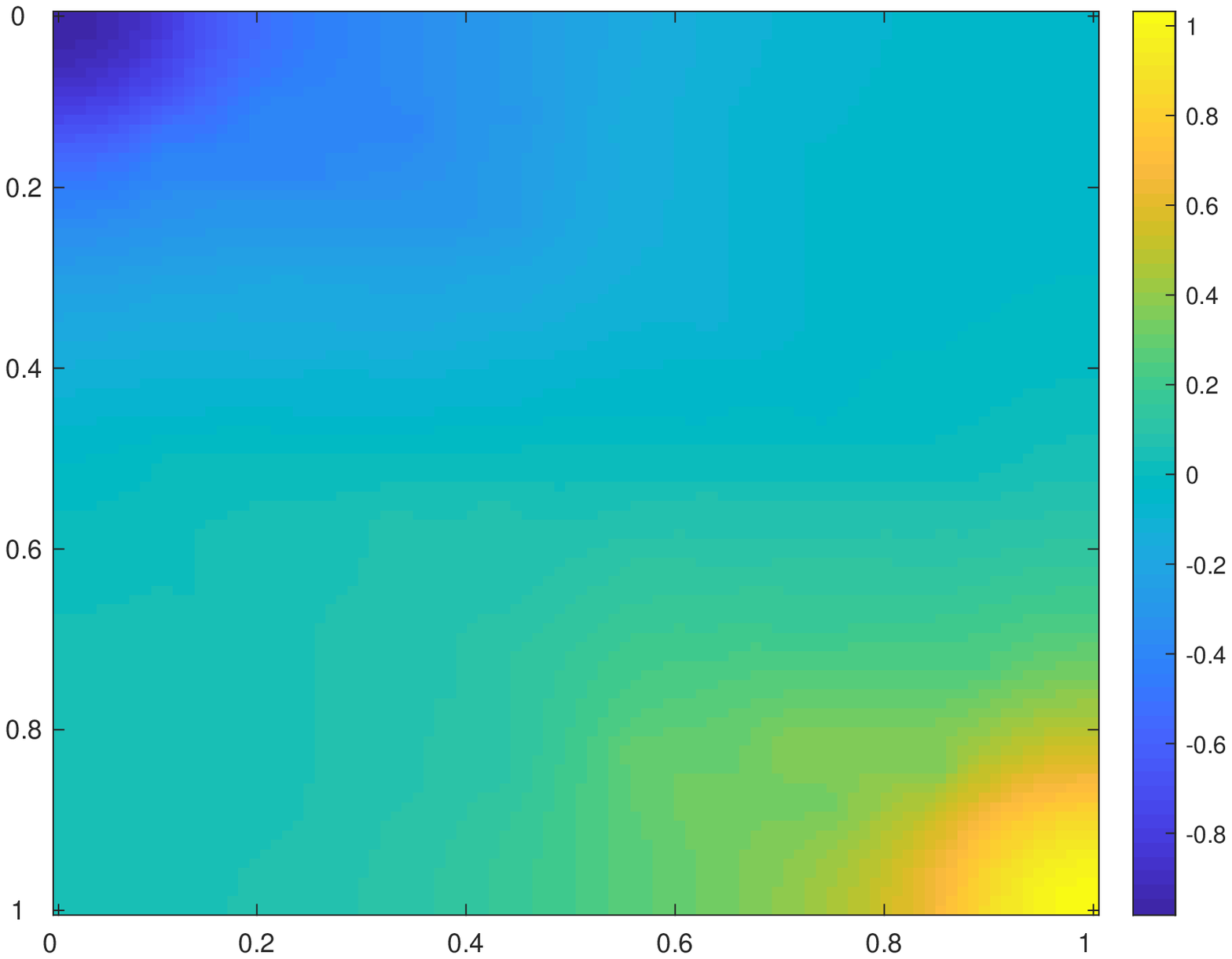}
	\label{fig:exp1_ref_p}}
	\subfigure[Velocity.]{
	\includegraphics[width=2.6in]{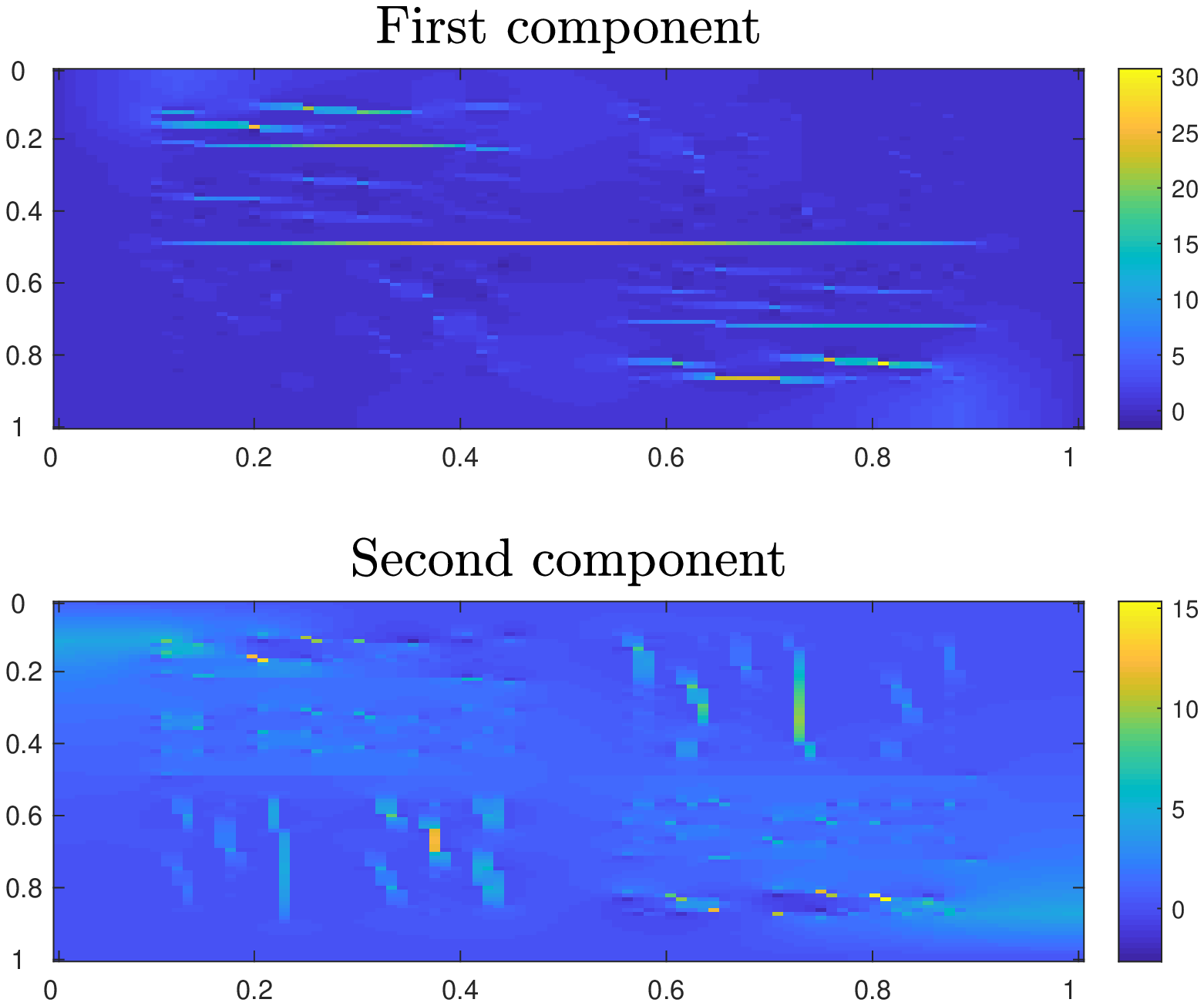}
	\label{fig:exp1_ref_v}}
	}
	\caption{Reference solution in Example \ref{exp:1}.}
	\label{fig:exp1_ref_sol}
	\end{figure}
	
	\begin{figure}[h!]
	\centering
	\mbox{
	\subfigure[Pressure.]{
	\centering
	\includegraphics[width=2.6in]{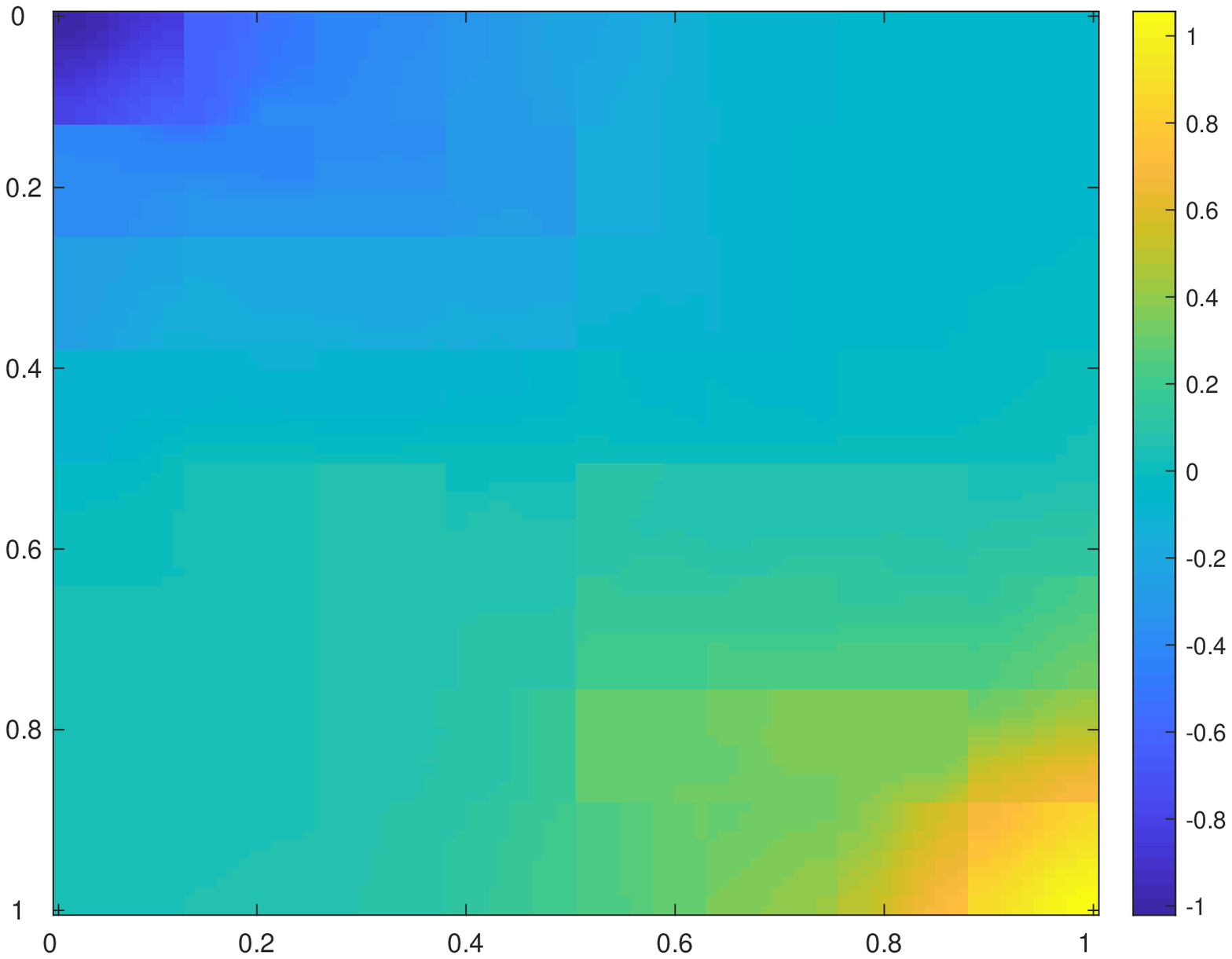}
	\label{fig:exp1_ms_p}}
	\subfigure[Velocity.]{
	\includegraphics[width=2.6in]{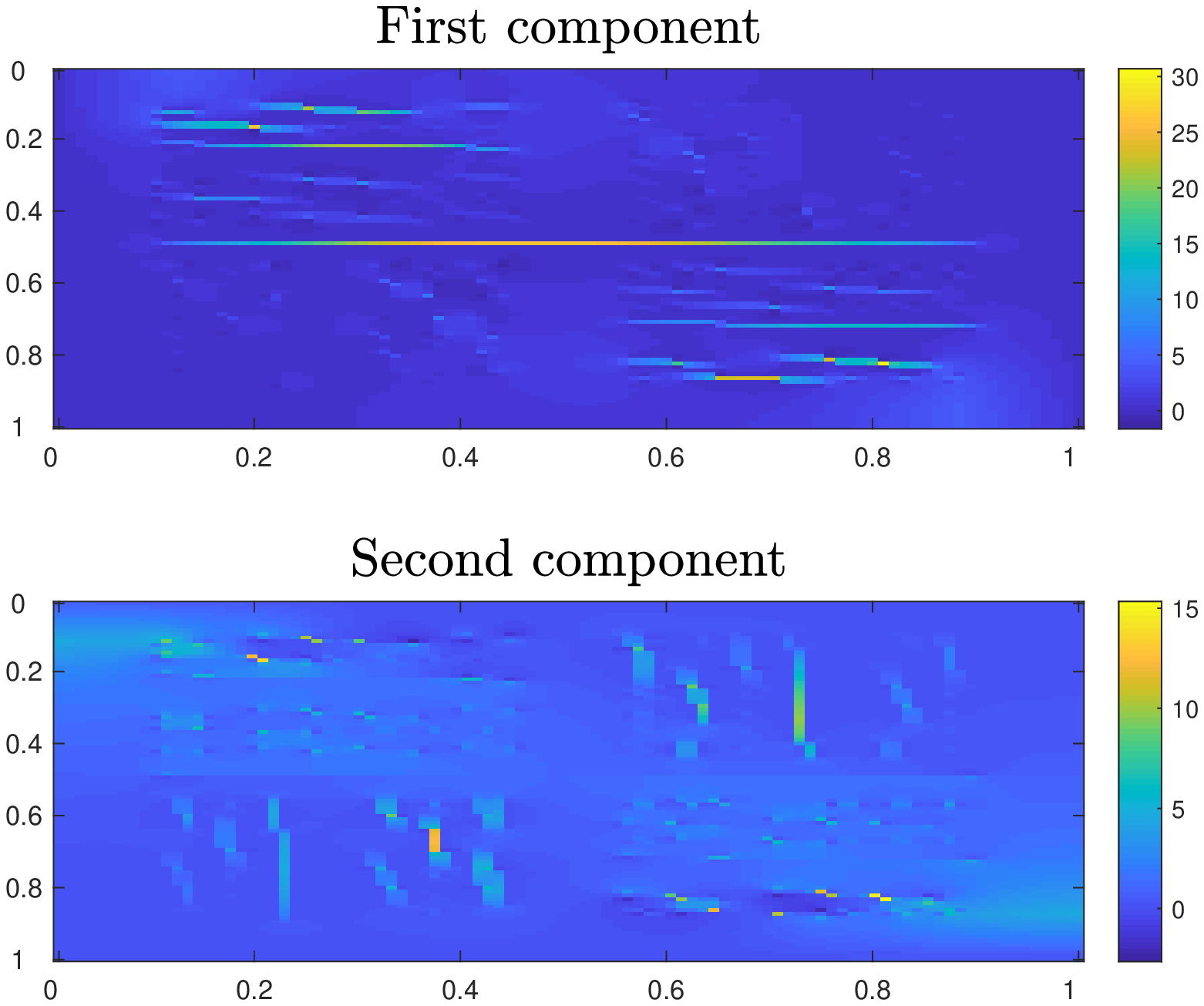}
	\label{fig:exp1_ms_v}}
	}
	\caption{Multiscale solution in Example \ref{exp:1} ($\theta = 1$; $m= 4$; $\tilde \ell = 2$).}
	\label{fig:exp1_ms_sol}
	\end{figure}

	\begin{table}[htbp!]
	\centering
	\begin{tabular}{ccccc}
	\hline 
		\revii{$J$} &  \revii{$m$} &  DOF & \revii{$e_u$}  & \reviii{CPU times (sec)} \\
	\hline 
		$3$ & $0$ & $192$ & \revii{$6.57590\%$}  & $-$ \\ 
		$3$ & $1$ & $241$ & \revii{$0.21606\%$}  & \reviii{4.6075} \\
		$3$ & $2$ & $290$ & \revii{$0.01968\%$}  & \reviii{4.3116} \\
		$3$ & $3$ & $339$ & \revii{$0.00165\%$}  & \reviii{4.3498} \\
		$3$ & $4$ & $388$ & \revii{$0.00013\%$}  & \reviii{4.2612} \\ 
	\hline
	\end{tabular}
	\caption{Uniform enrichment with $\theta= 1$, \revii{$\ell = 2$, and $\tilde \ell = 2$} (Example \ref{exp:1}).}
	\label{tab:exp_1_theta_1}
	\end{table}
	\begin{table}[htbp!]
	\centering
	\begin{tabular}{ccccc}
	\hline 
		\revii{$J$} &  \revii{$m$} &  DOF & \revii{$e_u$}  & \reviii{CPU times (sec)} \\
	\hline 
		$3$ & $0$ & $192$ & \revii{$6.57590\%$}  &$-$ \\
		$3$ & $1$ & $196$ & \revii{$2.89966\%$}  &\reviii{0.6599} \\ 
		$3$ & $2$ & $200$ & \revii{$0.89460\%$}  &\reviii{0.5291} \\ 
		$3$ & $3$ & $204$ & \revii{$0.66984\%$}  &\reviii{0.5327} \\ 
		$3$ & $4$ & $208$ & \revii{$0.38731\%$}  &\reviii{0.5876} \\ 
	\hline
	\end{tabular}
	\caption{Online adaptivity with $\theta= 0.1$, \revii{$\ell = 2$, and $\tilde \ell = 2$} (Example \ref{exp:1}).}
	\label{tab:exp_1_theta_0.1}
	\end{table}
	
	\begin{table}[htbp!]
	\centering
	\begin{tabular}{ccccc}
	\hline 
		\revii{$J$} &  \revii{$m$} &  DOF & \revii{$e_u$} & \reviii{CPU times (sec)}  \\
	\hline 
		$3$ & $0$ & $192$ & \revii{$6.57590\%$}  & $-$ \\
		$3$ & $1$ & $241$ & \revii{$0.00553\%$}  & \reviii{9.6672} \\ 
		$3$ & $2$ & $290$ & \revii{$5.37 \times 10^{-6}\%$} & \reviii{9.9367} \\	
	\hline
	\end{tabular}
	\caption{Uniform enrichment with $\theta= 1$, \revii{$\ell = 2$, and $\tilde \ell = 4$} (Example \ref{exp:1}).}
	\label{tab:exp_1_ol4}
	\end{table}
		
\end{example}

\begin{example}\label{exp:2}
In this example, we test the proposed method on another permeability field with high-contrast channels (see Figure \ref{fig:exp_2_kappa}). The mesh parameters are $T = 10$ and $n = 16$. The source function $f$ used in this example is defined as follows: 
$$ f(x)  = \left \{ \begin{array}{cc}
1 & x \in (0.2,0.4)^2, \\
-1 & x \in (0.7,0.9)^2, \\
0 & \text{otherwise},
\end{array} \right . $$
and thus the compatibility condition holds. In this example, \revii{the number of oversampling layers is $\ell = 1$. We set $\tilde \ell = 2$} to construct the online basis functions. 
\reviii{The computational time for offline computation is 12.9394 seconds in this example}. 
\revii{The reference and multiscale solutions of this example are provided below. See Figures \ref{fig:exp2_ref_sol} and \ref{fig:exp2_ms_sol}.} 

	\begin{figure}[h!]
	\centering
	\mbox{
	\subfigure[Pressure.]{
	\centering
	\includegraphics[width=2.6in]{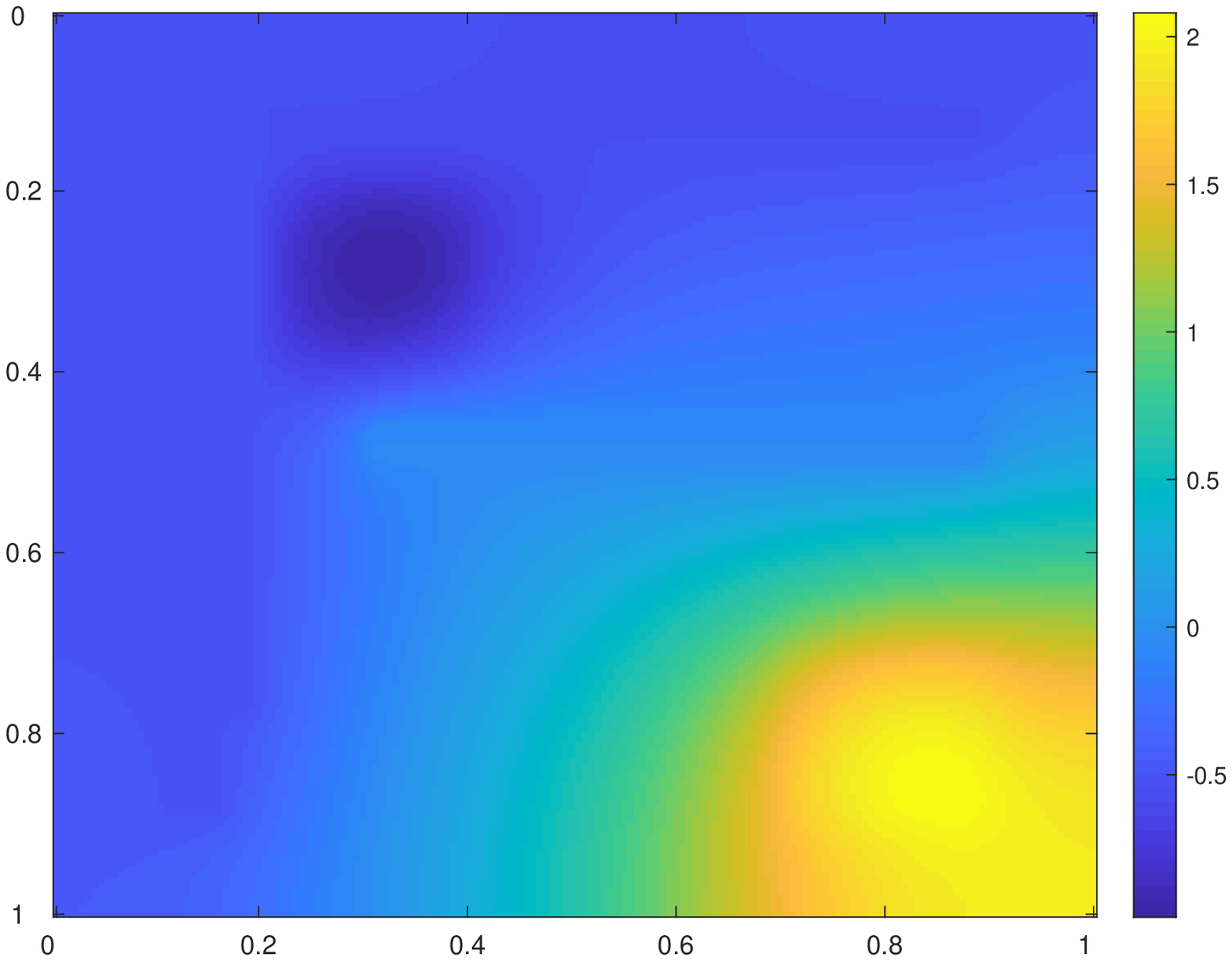}
	\label{fig:exp2_ref_p}}
	\subfigure[Velocity.]{
	\includegraphics[width=2.6in]{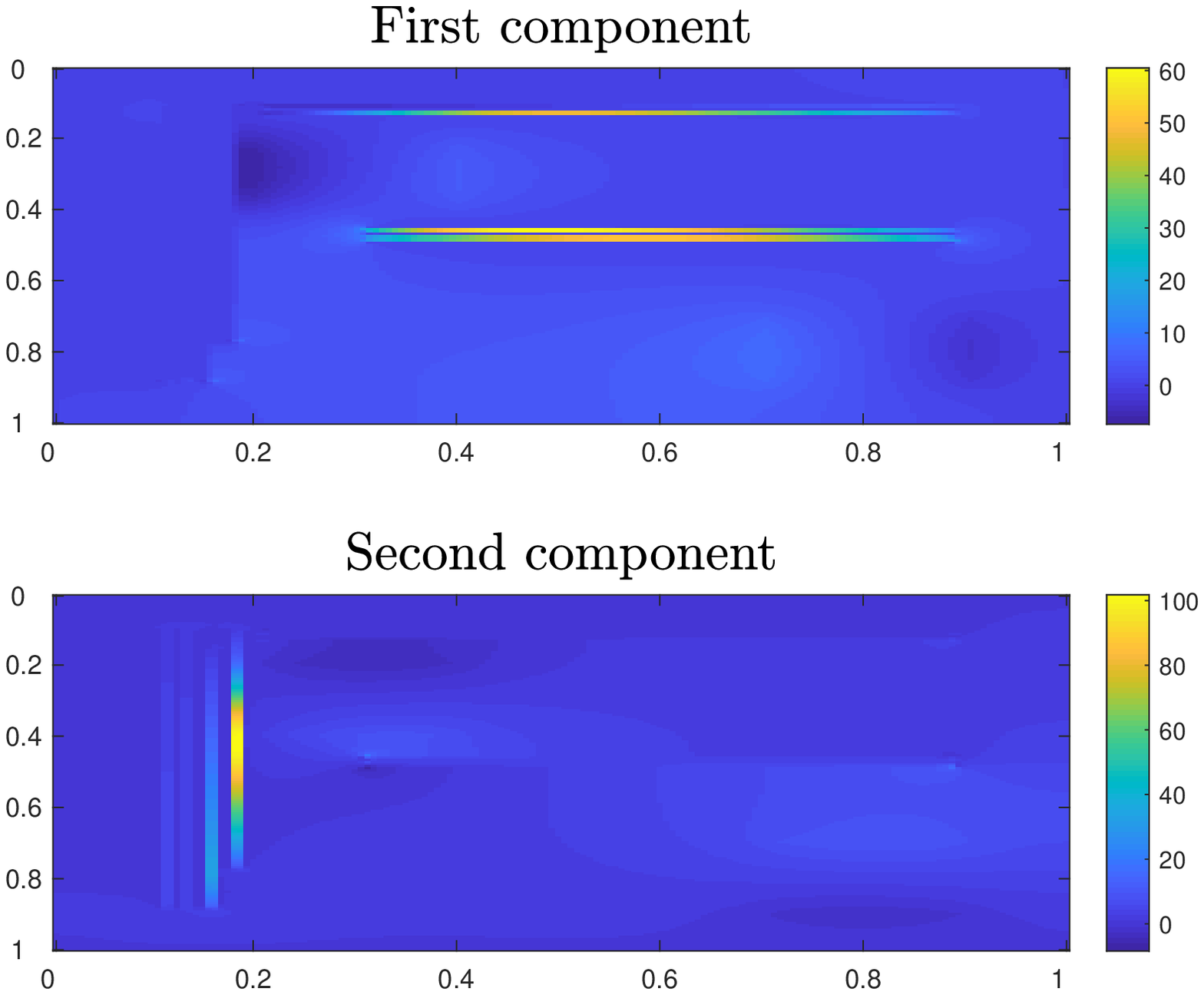}
	\label{fig:exp2_ref_v}}
	}
	\caption{Reference solution in Example \ref{exp:2}.}
	\label{fig:exp2_ref_sol}
	\end{figure}
	
	\begin{figure}[h!]
	\centering
	\mbox{
	\subfigure[Pressure.]{
	\centering
	\includegraphics[width=2.6in]{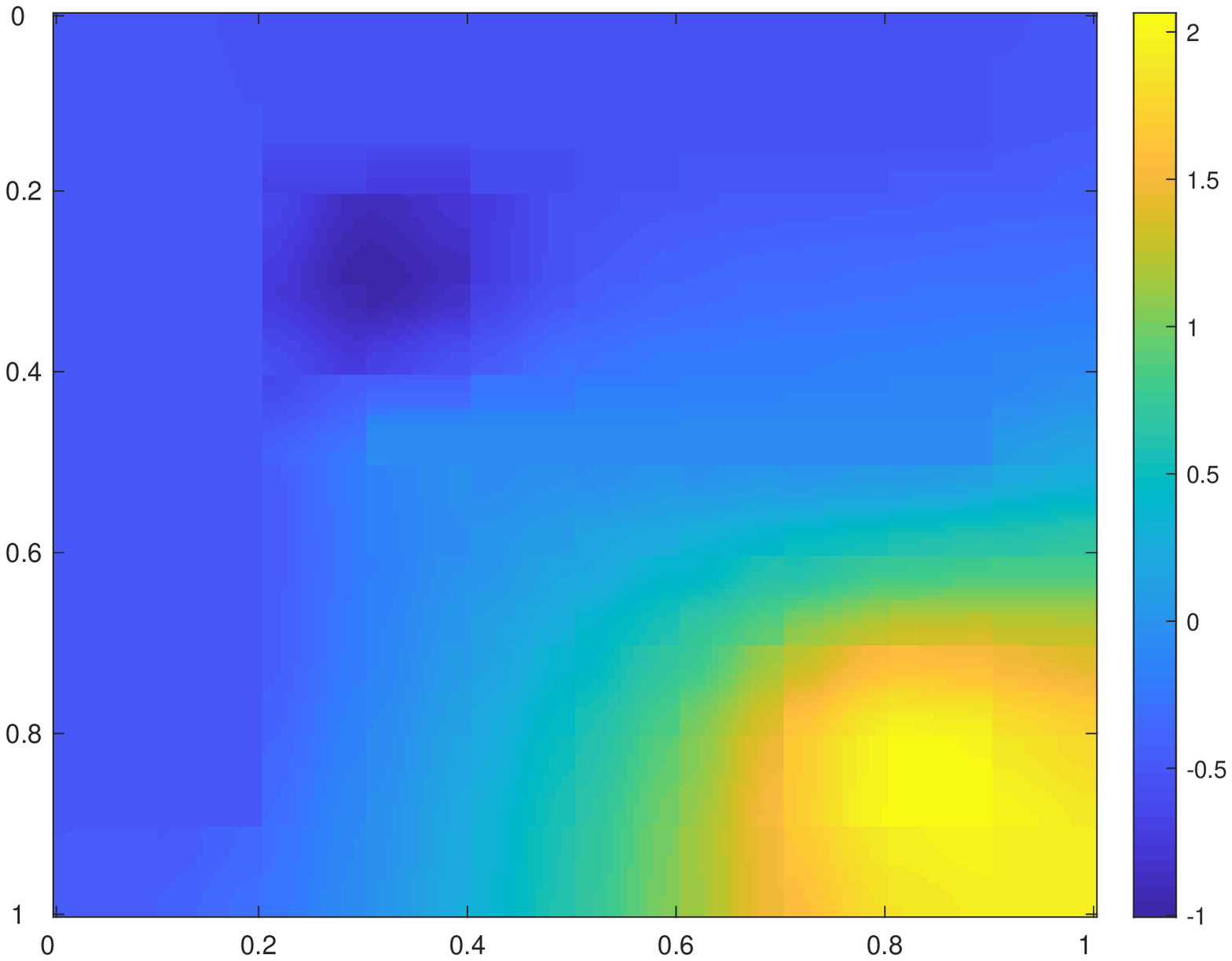}
	\label{fig:exp2_ms_p}}
	\subfigure[Velocity.]{
	\includegraphics[width=2.6in]{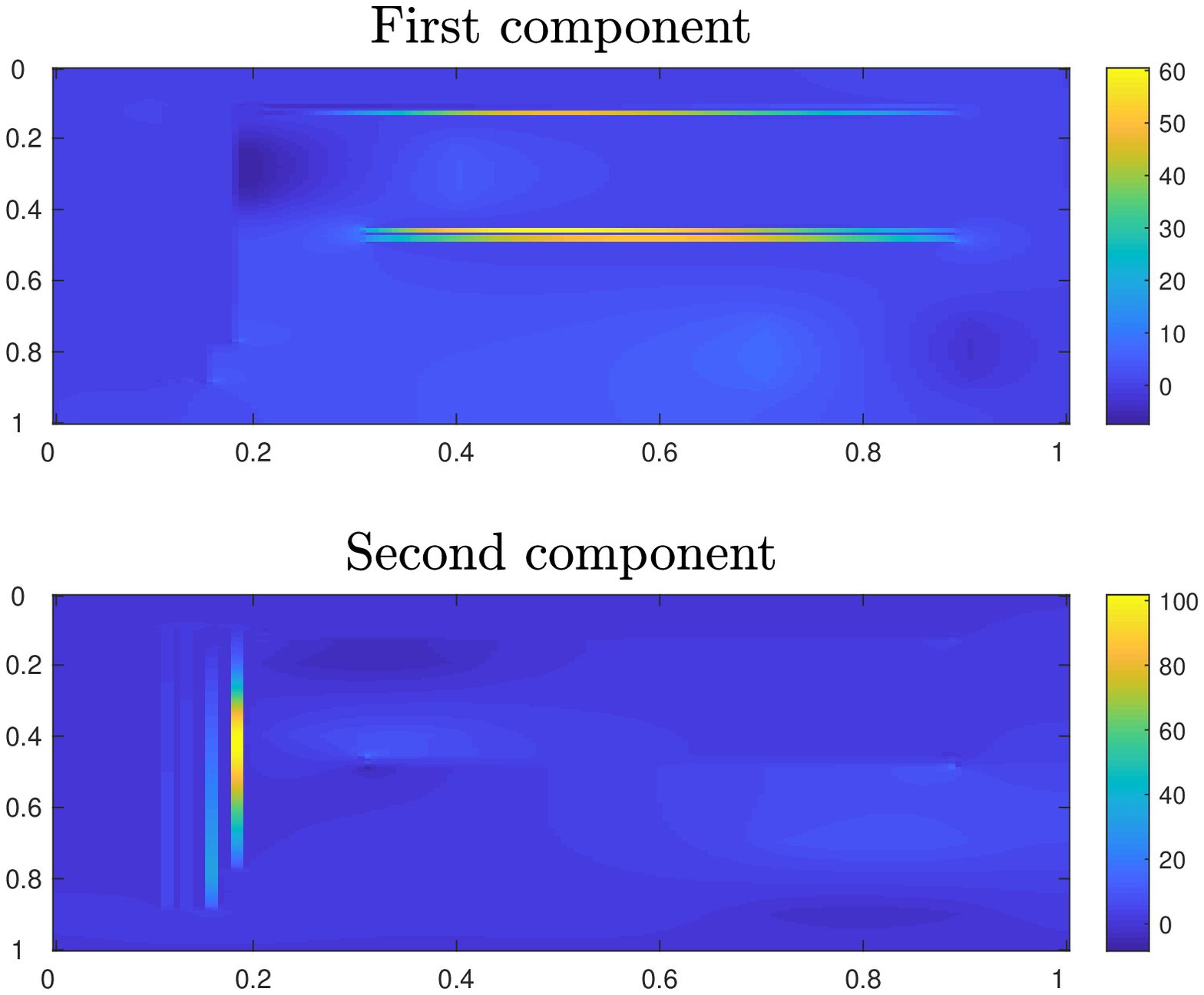}
	\label{fig:exp2_ms_v}}
	}
	\caption{Multiscale solution in Example \ref{exp:2} ($\theta = 1$; $m= 4$; $\tilde \ell = 2$).}
	\label{fig:exp2_ms_sol}
	\end{figure}
	
In Table \ref{tab:exp_2_theta_1}, we present the results with uniform enrichment. One may observe that the proposed online method can drive the energy error down fast even with a large error between the offline approximation and the fine-scale solution. Next, we test the performance of the method using online adaptivity. In Table \ref{tab:exp_2_theta_0.15}, the results with parameter $\theta = 0.15$ are presented and the convergence becomes slow comparing to the case of uniform enrichment, which confirms the analytical assertion. 

	\begin{table}[htbp!]
	\centering
	\begin{tabular}{ccccc}
	\hline 
		\revii{$J$} &  \revii{$m$} &  DOF & \revii{$e_u$} & \reviii{CPU times (sec)}  \\
	\hline 
		$3$ & $0$ & $300$ & \revii{$12.74319\%$} & $-$ \\ 
		$3$ & $1$ & $381$ & \revii{$1.70167\%$}  & \reviii{44.5020}  \\ 
		$3$ & $2$ & $462$ & \revii{$0.08707\%$}  & \reviii{46.3008} \\ 
		$3$ & $3$ & $543$ & \revii{$0.00581\%$}  & \reviii{50.2405} \\ 
	\hline
	\end{tabular}
	\caption{Uniform enrichment with $\theta= 1$, \revii{$\ell = 1$, and $\tilde \ell = 2$} (Example \ref{exp:2}).}
	\label{tab:exp_2_theta_1}
	\end{table}
	
	\begin{table}[htbp!]
	\centering
	\begin{tabular}{ccccc}
	\hline 
		\revii{$J$} &  \revii{$m$} &  DOF & \revii{$e_u$}  & \reviii{CPU times (sec)}\\
	\hline 
		$3$ & $0$ & $300$ & \revii{$12.74319\%$}  & $-$ \\ 
		$3$ & $1$ & $304$ & \revii{$3.79567\%$}  & \reviii{3.3094} \\ 
		$3$ & $2$ & $308$ & \revii{$1.13441\%$}  & \reviii{3.8641} \\ 
		$3$ & $3$ & $312$ & \revii{$0.71769\%$}  & \reviii{3.6112} \\ 
	\hline
	\end{tabular}
	\caption{Online adaptivity with $\theta= 0.15$, \revii{$\ell = 1$, and $\tilde \ell = 2$} (Example \ref{exp:2}).}
	\label{tab:exp_2_theta_0.15}
	\end{table}	
\end{example}

\begin{example}\label{exp:3}
\revii{We consider a benchmark SPE 10 test case \cite{christie2001tenth} with the grid parameters $T = n = 16$. The permeability field $\kappa$ for this example is sketched in Figure \ref{fig:kappa_exp3}. We take the source function defined as follows: 
$$ f(x)  = \left \{ \begin{array}{cc}
1 & x \in (0,1/16)^2, \\
-1 & x \in (15/16,1)^2, \\
0 & \text{otherwise}.
\end{array} \right . $$}

\revii{We take $\ell = \tilde \ell = 2$ in this example}. \reviii{The computational time for offline computation is 91.4529 seconds in this example}. 
\revii{The solution profiles of this example are reported in Figures \ref{fig:exp3_ref_sol} and \ref{fig:exp3_ms_sol}. 
The initial multiscale error ($m=0$) is relatively large (nearly $20\%$) in this case. One may observe that the proposed online enrichment is able to reduce the error of velocity efficiently. 
In particular, the relative energy error $e_u$ is quite small (around $0.002\%$) after three iterations. }

\begin{figure}[h!]
\centering
\includegraphics[width = 2.6in]{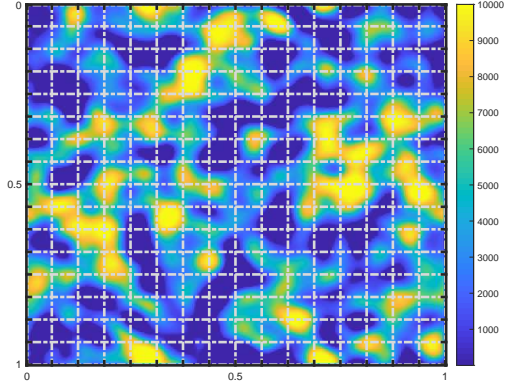}
\label{fig:kappa_exp3}
\caption{The permeability field $\kappa$ in Example \ref{exp:3}.}
\end{figure}

	\begin{figure}[h!]
	\centering
	\mbox{
	\subfigure[Pressure.]{
	\centering
	\includegraphics[width=2.6in]{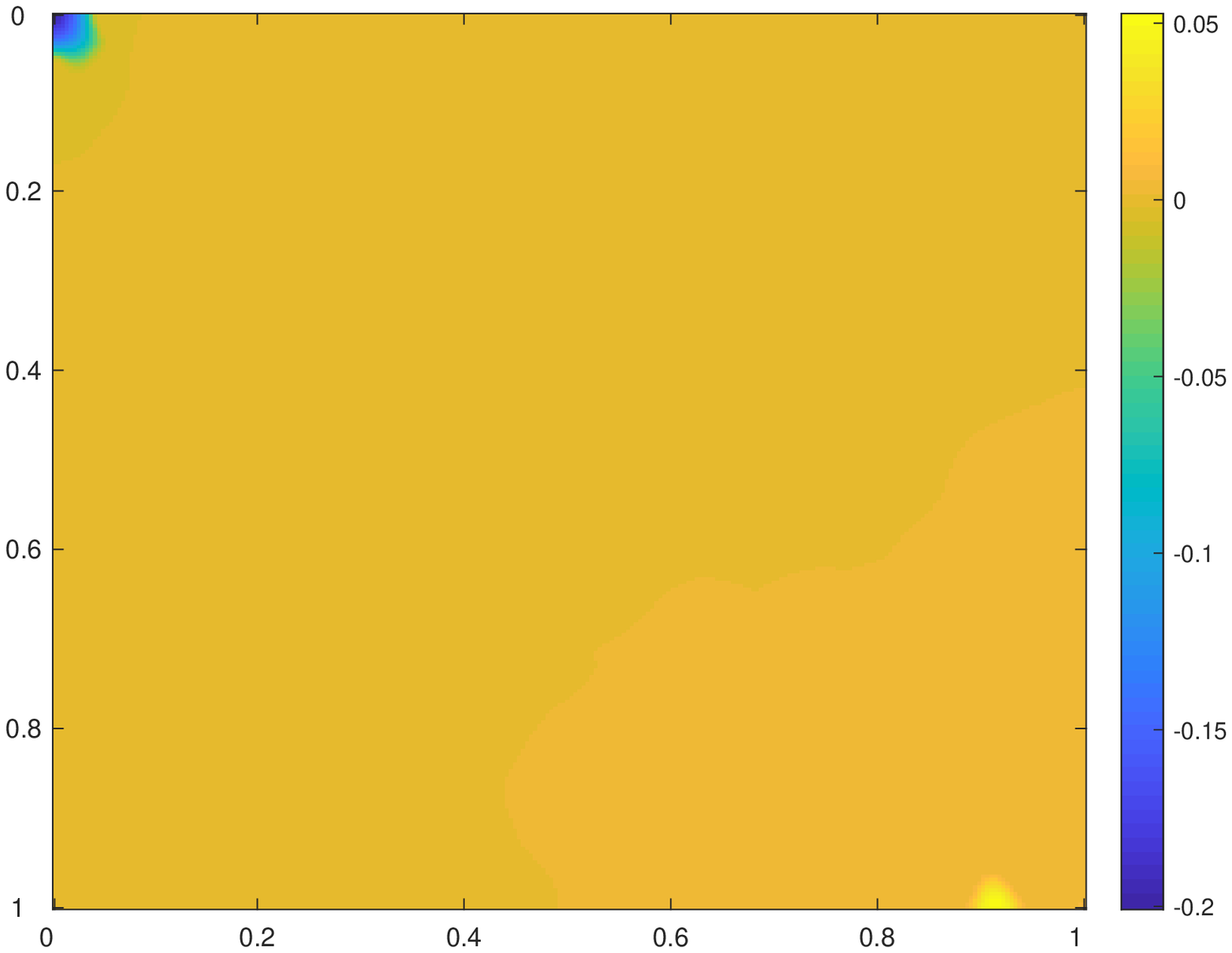}
	\label{fig:exp3_ref_p}}
	\subfigure[Velocity.]{
	\includegraphics[width=2.6in]{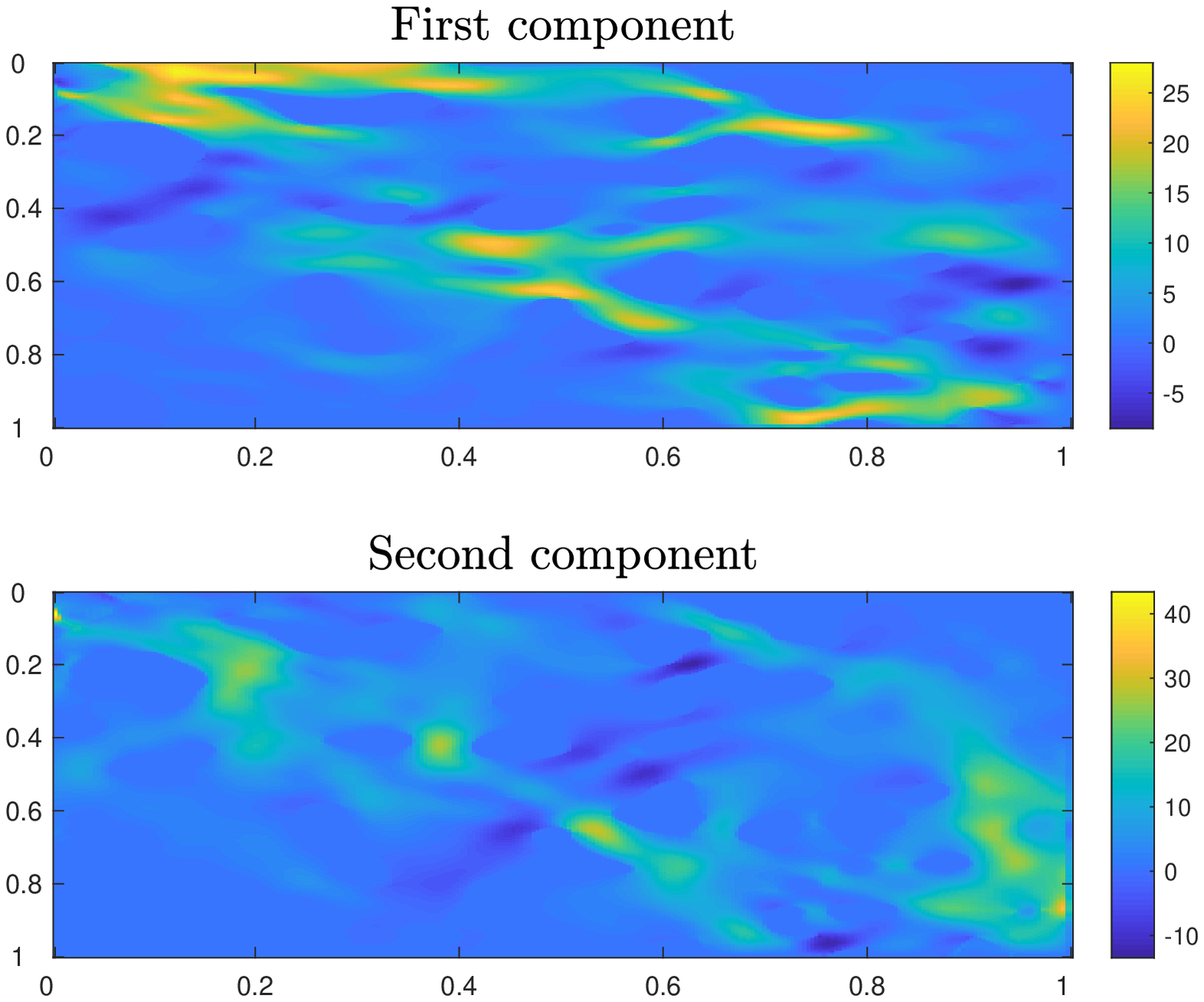}
	\label{fig:exp3_ref_v}}
	}
	\caption{Reference solution in Example \ref{exp:3}.}
	\label{fig:exp3_ref_sol}
	\end{figure}
	
	\begin{figure}[h!]
	\centering
	\mbox{
	\subfigure[Pressure.]{
	\centering
	\includegraphics[width=2.6in]{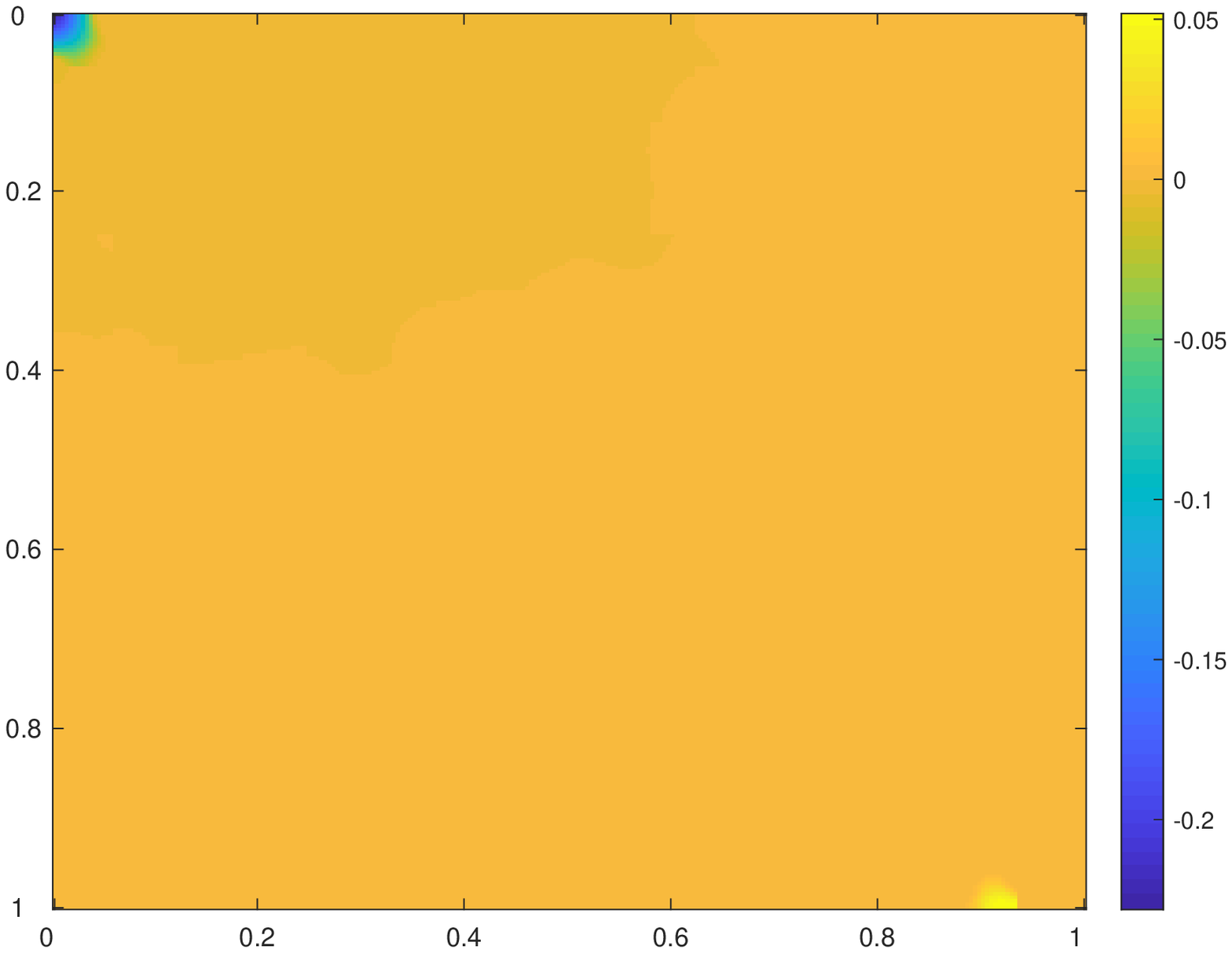}
	\label{fig:exp3_ms_p}}
	\subfigure[Velocity.]{
	\includegraphics[width=2.6in]{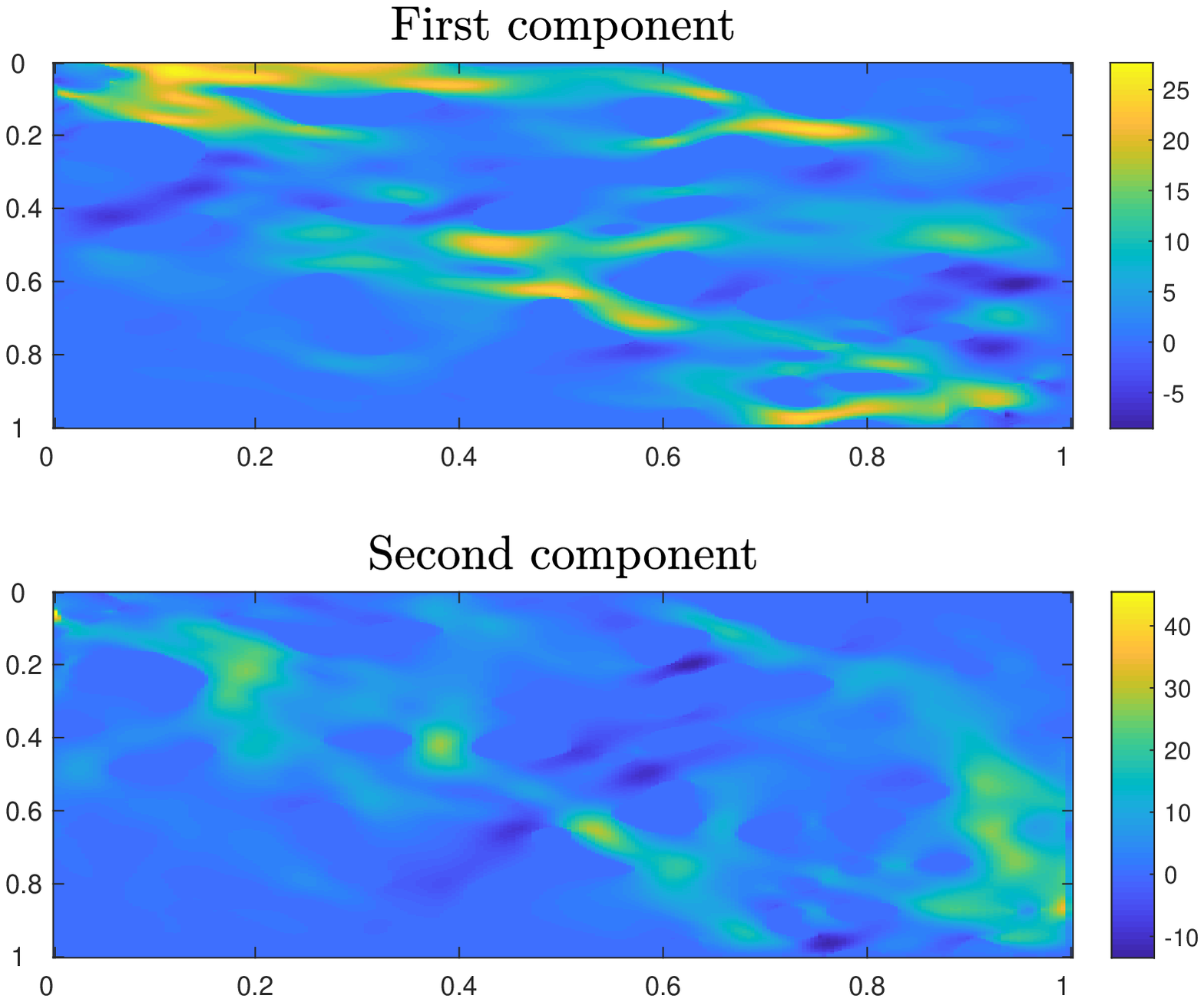}
	\label{fig:exp3_ms_v}}
	}
	\caption{Multiscale solution in Example \ref{exp:3} ($\theta = 1$; $m= 3$; $\tilde \ell = 2$).}
	\label{fig:exp3_ms_sol}
	\end{figure}

	\begin{table}[htbp!]
	\centering
	\begin{tabular}{ccccc}
	\hline 
		\revii{$J$} &  \revii{$m$} &  DOF & \revii{$e_u$}  & \reviii{CPU times (sec)}\\
	\hline 
		$3$ & $0$ & $768$ & \revii{$20.80919\%$} & $-$  \\ 
		$3$ & $1$ & $993$ & \revii{$0.42899\%$}    & \reviii{84.1748}\\ 
		$3$ & $2$ & $1218$ & \revii{$0.03002\%$}  & \reviii{84.0263}\\ 
		$3$ & $3$ & $1443$ & \revii{$0.00194\%$}  & \reviii{92.4949}\\ 
	\hline
	\end{tabular}
	\caption{Uniform enrichment with $\theta= 1$, \revii{$\ell = 2$, and $\tilde \ell = 2$} (Example \ref{exp:3}).}
	\label{tab:exp_3_theta_1}
	\end{table}
\end{example}

\section{Conclusion}
\label{sec:conclusion}
In this research, we propose an online adaptive strategy for CEM-GMsFEM in mixed formulation. The CEM-GMsFEM developed in \cite{chung2018constraint} provides a systematic approach to construct (offline) multiscale basis functions that give a mesh-dependent convergence rate, regardless of the heterogeneities of the media. In some applications, one may need to further improve the accuracy of the approximation without additional mesh refinement. In these cases, one needs to enrich the multiscale space by adding more basis functions in the online stage. The online basis functions for mixed CEM-GMsFEM are constructed by using the oversampling technique and the information of local residuals. Moreover, an adaptive enrichment algorithm is presented to reduce error in some selected regions with large residuals. The analysis of the method shows that the convergence rate depends on the constant of exponential decay and a user-defined parameter. Numerical experiments are provided to validate the analytical estimate.
\subsection*{Acknowledgement}
Eric Chung's work is partially supported by the Hong Kong RGC General Research Fund (project 14304217) and CUHK Direct Grant for Research 2017-18.
\bibliographystyle{plain}
\bibliography{references}

\end{document}